\documentclass[letterpaper,11pt]{amsart}

\usepackage[T1]{fontenc}
\usepackage[utf8]{inputenc}
\usepackage[english]{babel}
\usepackage{a4wide}
\usepackage{amsfonts}
\usepackage{amssymb,,amsmath}
\usepackage[dvipsnames]{xcolor}
\usepackage{pgfplots}
\usepackage{graphicx}
\usepackage{stmaryrd}
\usepackage{bbm} 
\usepackage{amsthm}
\usepackage{float} 
\usepackage{hyperref}
\usepackage{comment}
\hypersetup{
    colorlinks=true,       
    linkcolor=red,          
    citecolor=blue,        
    filecolor=magenta,      
    urlcolor=cyan           
}

\newtheoremstyle{note}{} {}{\itshape}{-6pt}{\bf}{. --}{ }{}

\theoremstyle{note}
\newtheorem{theorem}{Theorem}
\newtheorem{proposition}{Proposition}[section]
\newtheorem{definition}[proposition]{Definition}
\newtheorem{lemma}[proposition]{Lemma}

\numberwithin{equation}{section}

\newenvironment{remark}{{\noindent\bf Remark }}{}
\newenvironment{remarks}{{\noindent\bf Remarks }}{}

\usepackage{tikz}
\usetikzlibrary{calc}

\DeclareMathOperator{\diam}{diam}
\DeclareMathOperator{\dist}{dist}
\DeclareMathOperator{\Div}{div}

\DeclareMathOperator{\Int}{Int}

\DeclareMathOperator{\Ran}{Ran}

\DeclareMathOperator{\Ric}{Ric}

\DeclareMathOperator{\vol}{vol}

\title[Spectral estimates on hyperbolic surfaces]{Spectral estimates on hyperbolic surfaces and a necessary condition for observability of the heat semigroup on manifolds}
\author{Alix Deleporte, Marc Rouveyrol}
\address{Laboratoire de Math\'ematiques d'Orsay, Universit\'e Paris-Saclay, CNRS, B\^atiment~307, 91405 Orsay Cedex}
\email{alix.deleporte@universite-paris-saclay.fr}
\email{marc.rouveyrol@universite-paris-saclay.fr}

\begin{document}

\begin{abstract}
This article is a continuation of \cite{rouveyrol2023spectralH2}. We study the concentration properties of spectral projectors on manifolds, in connection with the uncertainty principle.
In \cite{rouveyrol2023spectralH2}, the second author proved an optimal uncertainty
principle for the spectral projector of the Laplacian on the
hyperbolic half-plane. The aim of the present work is to generalize
this condition to surfaces with hyperbolic ends. In particular, we
tackle the case of cusps, in which the volume of balls of fixed radius
is not bounded from below. We establish that spectral estimates hold
from sets satisfying a thickness condition, with a proof based on
propagation of smallness estimates of Carleman and
Logunov--Malinnikova type. We also prove the converse, namely the necessary
character of the thickness condition, on any smooth manifold with Ricci
curvature bounded from below.
\end{abstract}


\maketitle
\tableofcontents

\newpage

\section{Introduction and main result}

\subsection{Main result}

\subsubsection{Notation and main results}

On a non-compact Riemannian manifold $(M,g)$ without boundary, endowed with a smooth metric $g$, consider the Laplace--Beltrami operator $-\Delta_g$ with domain $C^\infty_c(M)$. The corresponding quadratic form is $q: D(q) \rightarrow \mathbb{R}$ given by $$D(q) = H^1(M) = \{u \in L^2(M), \nabla_g u \in L^2(M)\}, \quad q(u) = \int_M |\nabla_g u|_g^2 d\vol_g.$$ We denote by $-\Delta_g$ (with the convention $\Delta_g \leq 0$) the Friedrichs extension of $(-\Delta_g, C^\infty_c(M))$, with domain $D(-\Delta_g) = \{u \in H^1(M), \Delta_g u \in L^2(M)\}$. We refer to \cite{reedsimon1980methodsvol1} and \cite{lewin2022livrespectral} for spectral-theoretic considerations.

The Laplace--Beltrami operator is non-negative and self-adjoint, so
that its square root $\sqrt{-\Delta_g}$ is well-defined and admits a
spectral measure $dm_\Lambda$. Any function $u \in L^2(M)$ can be decomposed against this spectral measure and $\sqrt{-\Delta_g}$ admits a functional calculus: \begin{multline}
\label{eq: spectral measure}
u = \int_0^{+\infty} dm_\lambda u, \quad \phi(\sqrt{-\Delta_g})u =
\int_0^{+\infty} \phi(\lambda)dm_\lambda u\\
\langle \phi(\sqrt{-\Delta_g})u,
\varphi(\sqrt{-\Delta_g})u\rangle_{L^2(M)} = \int_0^{+\infty}
\phi(\lambda) \varphi(\lambda) \langle dm_\lambda u, u
\rangle\end{multline} for any $u \in L^2(M)$, $\phi, \varphi \in \mathcal{L}^\infty(\mathbb{R}^+)$. We then define the spectral projector on the frequency window $[0,\Lambda]$, $\Lambda > 0$, by \begin{equation}
    \label{eq: spectral projector} \Pi_\Lambda u = \mathbbm{1}_{\sqrt{-\Delta_g} \leq \Lambda} u = \int_0^\Lambda dm_\lambda u,
\end{equation} so that \begin{equation}
    \label{eq: functional calculus bound}
    \|\phi(\sqrt{-\Delta_g})\Pi_\Lambda u\|_{L^2(M)} \leq \sup_{\lambda \in [0,\Lambda]} |\phi(\lambda)| \|\Pi_\Lambda u\|_{L^2(M)}.
\end{equation}

The present article aims at studying uncertainty principles for spectrally localized functions in $\Ran(\Pi_\Lambda)$, in the specific case where $M$ is a hyperbolic surface, or a surface with hyperbolic ends. It is a natural sequel of the article \cite{rouveyrol2023spectralH2} by the second author, which dealt with the case of the hyperbolic half-plane $M = \left(\mathbb{H}, \frac{dx^2 + dy^2}{y^2}\right)$, for hyperbolic surfaces are quotients of $\mathbb{H}$ by suitable isometry groups. To avoid diluting the main theorems in geometric discussions, we next state the results of the article, and postpone definition of the precise geometric framework to the next subsection.

The uncertainty principle we prove takes the form of an inequality in the spirit of Logvinenko and Sereda's spectral projector estimates \cite{logvinenkosereda1974equivalent,kovrizhkin2000fourier,kovrizhkin2001relatedtoLog-Ser}. Precisely, we say that a Borel subset $\omega \subset M$ satisfies a \textit{spectral estimate} if there exists some constant $C(\omega)$, independent of $u$ and of the frequency threshold $\Lambda$, such that \begin{equation}
    \label{eq: spec} \tag{Spec}
    \forall u \in L^2(M), \forall \Lambda > 0, \quad \|\Pi_\Lambda u\|_{L^2(M)} \leq C e^{C\Lambda} \|\Pi_\Lambda u\|_{L^2(\omega)}.
\end{equation}

In \cite{rouveyrol2023spectralH2}, we proved that in the case $M = \mathbb{H}$, the spectral estimate is equivalent to the so-called \textit{thickness} condition:

\begin{definition}
\label{ass: renormalized thickness condition}
Let $R > 0$. We say that $\omega \subset M$ is \emph{thick at scale $R$} (or $R$-thick) if there exists some volume ratio $\delta(R) > 0$ such that \begin{equation}
    \label{eq: thickness} \tag{Thick}
    \forall z \in M, \frac{\vol(B_z(R) \cap \omega)}{\vol(B_z(R))} \geq \delta.
\end{equation} We shall say that $\omega$ is \emph{thick} if there exists some $R > 0$ such that $\omega$ is $R$-thick.
\end{definition}

The first of our results states that thickness of $\omega$ is sufficient for a spectral estimate to hold:

\begin{theorem}
    \label{theo: main result, sufficient part}
    Assume $M$ is a surface with hyperbolic ends. Take $\omega$ a Borel subset of $M$. If $\omega$ satisfies \eqref{eq: thickness} for some parameters $(R,\delta)$, then $\omega$ satisfies \eqref{eq: spec} for some constant $C(\omega)$.
\end{theorem}

Section \ref{sec: proof sufficient condition} is dedicated to
the proof of Theorem \ref{theo: main result, sufficient part}. The next theorem states the converse of Theorem \ref{theo: main result, sufficient part}, namely that the thickness assumption is also necessary for a spectral estimate to hold. Only a Ricci curvature lower bound is required, making the assumptions on $M$ much weaker.

\begin{theorem}
    \label{theo: main result, necessary part}
    Let $(M, g)$ be a complete, boundaryless, smooth
    Riemannian manifold, not necessarily compact, of arbitrary dimension, with Ricci curvature bounded from below: \begin{equation}
        \label{eq: Ricci curvature bound, introduction}
        \Ric \geq -Kg, \quad K > 0.
    \end{equation} in the sense of quadratic forms. If $\omega$ satisfies \eqref{eq: spec}, then $\omega$ is thick at scale $R$ for some sufficiently large $R > 0$.
\end{theorem}

Theorem \ref{theo: main result, necessary part} is proved in Section
\ref{sec: proof necessary condition}. A complete
bibliographical discussion of previous similar results and of elements
of proof will come in the next subsections of this introduction,
but we already make some comments about our main results. They are the object of the following remarks.

\bigskip

\begin{remarks}
    \begin{enumerate}
        \item As explained in the next subsection, surfaces with hyperbolic ends do satisfy the Ricci curvature lower bound $\Ric \geq -Kg$, $K > 0$. Consequently, Theorem \ref{theo: main result, necessary part} implies that the converse of Theorem \ref{theo: main result, sufficient part} holds.
        \item In both Theorem \ref{theo: main result, sufficient part}
          and Theorem \ref{theo: main result, necessary part}, the
          metric $g$ is assumed to be smooth. It is likely that this
          assumption could be weakened in Theorem \ref{theo: main
            result, necessary part}, as the natural framework in which
          the assumption \[\Ric(g) + Kg \geq 0\] is well-defined is
          that of Lipschitz metrics (with a distributional definition
          of the Ricci tensor). $\Ric(g) + Kg$ would then be a
          positive measure, implying no gain of regularity on $g$
          even in harmonic charts. However, the proof of Theorem \ref{theo: main result, necessary part} relies on heat kernel estimates which have only been proved for smooth metrics, to the best of the authors' knowledge.
        \item One interesting feature of surfaces with hyperbolic ends is that the denominator $\vol(B_z(R))$ which appears in the denominator of \eqref{eq: thickness} can tend to zero, in the asymptotic in which $z$ is sent to the end of a cusp. Thus, unlike in preceding works \cite{burqmoyano2021spectralestimates,rouveyrol2023spectralH2}, there is no uniform lower bound for the volume of $\omega$ inside a given ball of radius $R$.
        \item In a recent article \cite{lebalchmartin2024quantitative}, Kévin Le Balc'h and Jérémy Martin prove an equivalent of Theorem \ref{theo: main result, sufficient part} for operators of the form $-\Div(A(x) \nabla) + V(x), V \in L^\infty$, on the full flat space. Their proof, like ours, is based on Logunov--Malinnikova estimates for the gradients of solutions to elliptic equations \cite{logunovmalinnikova2018qtttve_propag_smallness}, in the spirit of Burq and Moyano \cite{burqmoyano2021spectralestimates,burqmoyano2023propagationheat}.
        It would be interesting to combine the two techniques to
        treat operators of the form $-\Delta_g + V$, with $V \in L^\infty(M)$.
        \item The following volume doubling estimate holds on any
        Riemannian manifold with Ricci curvature bounded from below
        and is a consequence of the Bishop--Gromov inequality        \cite{cheegergromovtaylor1982kernelestimates,gallothulinlafontaine2004book}:
        there exists a constant $C_D > 0$ such that for any $z \in
        M, R > 0$, \begin{equation}\label{eq: volume doubling intro}\vol(B_z(2R)) \leq C_D e^{C_D R}
        \vol(B_z(R)).\end{equation} The doubling constant $C_D$ depends only on
        the dimension of the manifold and the curvature lower
        bound. The $e^{C_D R}$ factor is obtained by
        comparison theorems between the growth of balls on manifolds
        with Ricci curvature bounded from below and on the
        hyperbolic space with constant curvature equal to the lower
        bound, on which it is the correct asymptotic for large $R$. Since the critical mass ratio $\delta$ of Definition \ref{ass: renormalized thickness condition} is allowed to depend on $R$, $R_0$-thickness of $\omega$ for some $R_0 > 0$ implies $R$-thickness for any $R > R_0$.
        \item In a manifold with infinite diameter, a thick set cannot be bounded.
        \item By the volume doubling estimate above, if $\omega$ is $R$-thick and $K$ is a compact subset of $\omega$, then $\omega \setminus K$ remains thick (typically at scale $\diam(K) + 2R$).
    \end{enumerate}
\end{remarks}

\subsubsection{Hyperbolic surfaces and hyperbolic ends}
\label{sec: explanations hyperbolic}

We now give all the relevant definitions related to the geometry of hyperbolic surfaces and surfaces with hyperbolic ends. The key elements are the classification of hyperbolic ends given in Theorem \ref{theo: classification of hyperbolic ends} and the realizations of funnels and cusps as quotients of appropriate subsets of the hyperbolic half-plane (depicted in Figure \ref{fig: funnel and cusp in H2}). The presentation follows Chapter 2 of David Borthwick's excellent book on the topic \cite{borthwick2016book}, keeping only information relevant to the article. For complementary material on spectral theory of hyperbolic surfaces, we refer to \cite{borthwick2016book,bergeron2011livre} and the references therein.

The hyperbolic upper half-plane will be denoted $\mathbb{H} = \{(x,y) \in \mathbb{R}^2, y > 0 \}$. To alleviate notations, we will often denote elements of $\mathbb{H}$ as $z=x + iy$. The hyperbolic metric on $\mathbb{H}$ is $g_{\mathbb{H}} = \frac{dx^2 + dy^2}{y^2}$. The corresponding Riemannian measure is $d\vol_{\mathbb{H}} = \frac{dxdy}{y^2}$ and the Laplace--Beltrami operator on $\mathbb{H}$ is given by $$\Delta_{\mathbb{H}} = y^2 \left(\partial_x^2 + \partial_y^2\right).$$ Geodesics for the metric $g_{\mathbb{H}}$ are vertical straight lines and arcs of circles centered on the $x$-axis. Geodesic balls are Euclidean balls with center shifted upwards with respect to the hyperbolic center. Precisely, given a point $z = x+iy \in \mathbb{H}$ and a radius $R > 0$, we have \begin{equation}
    \label{eq: balls in hyperbolic half-plane}
    B_z^{\mathbb{H}}(R) = B^{euc}_{x+ i y\cosh(R)}(y\sinh(R)).
\end{equation} where $B^{\mathbb{H}}$ denotes a geodesic ball and $B^{euc}$ denotes a Euclidean ball.

The group of isometries of $(\mathbb{H}, g_{\mathbb{H}})$, known as Möbius transformations, is isomorphic to the group of $2 \times 2$ real matrices with determinant 1, quotiented by $\{\pm I_2\}$. The common notation for this group is $PSL_2(\mathbb{R})$. Given a matrix $$T = \begin{pmatrix}
    a & b \\ c & d
\end{pmatrix}$$ in $PSL_2(\mathbb{R})$, the corresponding isometry is the complex function $$z = x+iy \mapsto \frac{az + b}{cz + d}.$$ Discrete subgroups of $PSL_2(\mathbb{R})$ are called \textit{Fuchsian}. They are exactly the groups of isometries acting properly discontinuously on $\mathbb{H}$.

We call a \textit{hyperbolic surface} any complete, 2-dimensional Riemannian manifold with Gaussian curvature constant equal to $-1$. As a consequence of Hopf's theorem on the classification of manifolds with constant sectional curvature \cite[Theorem 11.12 and Corollary 11.13]{lee1997bookriemannian}, any hyperbolic surface $X$ is isometric to a quotient of $\mathbb{H}$ by some Fuchsian group $\Gamma$.

In \cite{rouveyrol2023spectralH2}, the second author proved the equivalence between \eqref{eq: thickness} and \eqref{eq: spec} for $\Delta_\mathbb{H}$. A main contribution of the present work is to generalize this to geometrically finite hyperbolic surfaces, the definition of which we now give. A \textit{fundamental domain} $\mathcal{F} \subset \mathbb{H}$ for a Fuchsian group $\Gamma$ is a closed subset satisfying \[\bigcup_{T \in \Gamma} T\mathcal{F} = \mathbb{H}, \quad \forall T \in \Gamma \setminus \{I\}, \Int(\mathcal{F}) \cup \Int(T\mathcal{F}) = \varnothing.\] We say that a Fuchsian group is geometrically finite if it admits a fundamental domain which is a geodesically convex polygon with a finite number of sides. This is equivalent to the hyperbolic surface $\mathbb{H}/\Gamma$ having finite genus (number of holes) and a finite number of ends.

\begin{figure}[t]
    \centering
	\includegraphics[width=0.9\textwidth]{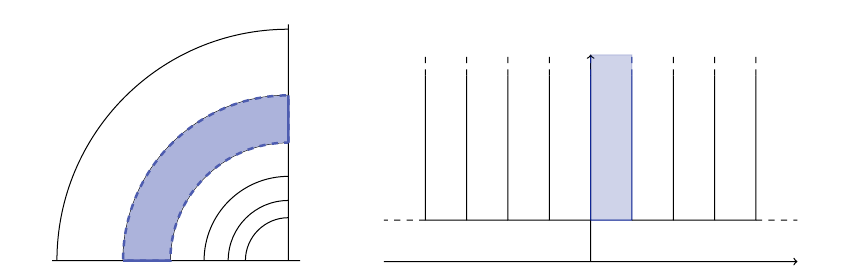}
    \caption{Left: fundamental domain for a funnel (delimited by three geodesic arcs, the vertical one of length $l$) in the hyperbolic half-plane $\mathbb{H}$ and its images by the cyclic Fuchsian group $\langle z \mapsto e^l z \rangle$. Right: fundamental domain $[0,1]~\times~[a,+\infty[$ for a cusp in the hyperbolic half-plane $\mathbb{H}$ and its images by the Fuchsian group $\langle z \mapsto z+1\rangle$.}
    \label{fig: funnel and cusp in H2}
\end{figure}

The structure of geometrically finite hyperbolic surfaces is
well-understood.
$\mathbb{H}$ and its quotients by cyclic Fuchsian groups are called \textit{elementary} surfaces. The classification result for non-elementary geometrically finite hyperbolic surfaces is the following \cite[Theorem 2.23]{borthwick2016book}:
 \begin{theorem}[Classification of hyperbolic ends]
    \label{theo: classification of hyperbolic ends}
    Any non-elementary, geometrically finite hyperbolic surface $M$ admits a compact region $K$ such that $M \setminus K$ is a finite, disjoint union of cusps and funnels. 
\end{theorem} 
Namely, (non-elementary, geometrically finite) hyperbolic surfaces 
 can
be decomposed as a disjoint union of elements of the three
  following type:
  \begin{itemize}
  \item a compact core, known as the
truncated Nielsen region;
\item funnels;
\item cusps.
  \end{itemize}
Here, given a hyperbolic cylinder $\mathbb{H} / \langle z
\mapsto e^l z \rangle$, where $l > 0$ is the length
of its delimiting geodesic $\{z\in \mathbb{H},{\rm Im}(z)=0\}$, a \emph{funnel} is a half-cylinder, namely
a connected component of the cylinder minus the geodesic.

Given a parabolic cylinder $\mathbb{H}/\langle z \mapsto z + 1 \rangle$, a \emph{cusp} is ``the small end'' of the cylinder. More
precisely, denoting $\mathbb{S}^1_\theta = \mathbb{R}/\mathbb{Z}$ the circle of length 1, cusps are isometric to
$\mathbb{S}^1_\theta \times (a,+\infty)_y$ for some $a > 0$ depending
on the cusp, endowed with the metric $\frac{d\theta^2 + dy^2}{y^2}$.

We depict fundamental domains for funnels and cusps in Figure
\ref{fig: funnel and cusp in H2}. Note that the local injectivity
radius at a point (the largest radius of a geodesic ball centered on
that point and onto which the exponential map is injective) goes to
zero asymptotically at the end of a cusp but grows to infinity in a
funnel. In particular, the volume of balls of fixed radius centered at
a given point tends to $0$ towards the end of a cusp but is
asymptotically constant towards the end of a funnel.

In the article, funnels will be denoted $F$ or $F_j$ and cusps will be denoted $\mathcal{C}$ or $\mathcal{C}_j$.

From now on, we will always imply that the geometric finiteness property holds when mentioning hyperbolic surfaces.

The nature of the compact core does not intervene in the proof of our
results. Thus, though we are motivated by hyperbolic surfaces, our
main result is stated for \emph{surfaces with hyperbolic ends}. The precise framework for Theorem \ref{theo: main result, sufficient part} is thus the following.

\begin{definition}
    \label{def: surface with hyperbolic ends}
    We say that a two-dimensional Riemannian manifold $M$ endowed with a smooth metric $g$ is a surface with hyperbolic ends if there exists some compact region $K \subset M$, still called the compact core, such that $M \setminus K$ is a finite union of funnels and cusps.
  \end{definition}

We insist on the fact that in this definition, no geometric assumption is made in the compact core $K$.

We end this overview by recalling the behavior of the spectrum of the Laplace--Beltrami $\Delta_g$ on a hyperbolic surface. Note that cusps have finite volume while the volume of a funnel is infinite. For compact hyperbolic surfaces, like on any compact manifold, the spectrum is made of discrete eigenvalues and there exists an orthonormal basis of eigenvectors in $L^2(M)$. We shall not discuss this case further as it was covered by Burq and Moyano in \cite{burqmoyano2023propagationheat}, in the case of Dirichlet or Neumann boundary conditions, $W^{2,\infty}$ manifolds and Lipschitz metrics. For non-compact, finite-volume hyperbolic surfaces (meaning they have no funnels and a finite number of cusps), the Laplace--Beltrami operator has absolutely continuous spectrum $[\frac{1}{4},\infty)$, a discrete spectrum consisting of finitely many eigenvalues in $[0, \frac{1}{4})$, and may have up to infinitely many embedded eigenvalues in $[\frac{1}{4},\infty)$ \cite[Theorem 3.12]{borthwick2016book}. These embedded eigenvalues are very unstable, in the sense that they are destroyed by certain deformations of $\Gamma$ \cite{phillipssarnak1985cusps}. There are no such embedded eigenvalues in the infinite-area case \cite[Proposition 7.5]{borthwick2016book}.

Lastly, we emphasize that Theorem \ref{theo: main result, sufficient part} for hyperbolic surfaces cannot be deduced immediately from the main result of \cite{rouveyrol2023spectralH2} on the full half-plane $\mathbb{H}$, as the $\Gamma$-periodic extension of a function in $L^2(\mathbb{H}/\Gamma) \setminus \{0\}$ does not belong to $L^2(\mathbb{H})$. Should there exist a strategy based on spectral inequalities on the full half-plane and periodic extensions from $\mathbb{H}/\Gamma$ to $\mathbb{H}$, we believe that our local-to-global, chart-based method for proving Theorem \ref{theo: main result, sufficient part} is more robust in the perspective of generalizations to a larger class of manifolds.

\subsection{Context}

This bibliographical overview is divided into two parts. First, we present a hopefully not too partial history of spectral estimates and their applications, with an attempt to cover the most recent advances in the field. Next, we cover some bibliography on the propagation of smallness inequality toolbox that is used throughout the article and its connection to spectral estimates.

\subsubsection{Spectral estimates on manifolds, motivations and applications}

The study of spectral estimates of the form \eqref{eq: spec} from
thick sets goes back at least to Logvinenko--Sereda and Kacnel'son \cite{logvinenkosereda1974equivalent,kacnelson1973equivalentnorms}, in connection with uncertainty principles for the Fourier transform. Their methods are complex-analytic and were improved by Nazarov and Kovrizhkin, among others \cite{nazarov1994local,kovrizhkin2000fourier,kovrizhkin2001relatedtoLog-Ser,kovrizhkin2003uncertainty}. In particular, Kovrizhkin made explicit the dependence of the constant $Ce^{C\Lambda}$ appearing in \eqref{eq: spec}, in terms of the thickness parameters $\delta$ and $R$. 

In their seminal article \cite{lebeaurobbiano1995controlechaleur}, Lebeau and Robbiano establish that spectral estimates allowed to prove controllability results for the heat equation on compact manifolds, both with and without boundary. In the boundary-free case, they prove exact null-controllabilty of the heat equation for $L^2$ initial data and arbitrarily short time when the sensor set $\omega$ is open. The equivalence between controllability and observability and the link with spectral estimates are developed in an abstract setting by Miller \cite{miller2010directlebeaurobbiano}, with explicit dependence of the controllability and observability costs. Unlike in \cite{lebeaurobbiano1995controlechaleur}, where exact controllability is shown directly and observability is given as a corollary, most recent controllability results for the heat equation in the spirit of Lebeau and Robbiano now prove observability directly. The telescoping series method of Phung and Wang \cite{phungwang2013observability} provides a robust proof that spectral estimates imply observability from positive-measure time sets via appropriate interpolation estimates. Subsequent developments of this method for Borel sensor sets $\omega$ on the flat space include \cite{apraizescauriaza2013nullcontrol,apraizescauriazawangzhang2014observability}. The sharp character of the thickness condition for observability of the heat equation on the flat space $\mathbb{R}^d$ has been proved simultaneously and independently in \cite{egidiveselic2018controlheat} and \cite{wangwangzhangzhang2019spectralineq_heat}. For additional results on spectral estimates and controllability of parabolic equation from thick sets in non-compact subdomains of the flat space we refer to \cite{egidiseelmann2021Logv-Seredatype}, the review article \cite{egidinakicseelmann+2020heatequation} and the references therein.

The Lebeau--Robbiano strategy for controllability generalizes to other parabolic equations beyond heat, including linear thermoelasticity systems \cite{lebeauzuazua1998thermoelasticity,benabdallahnaso2002thermoelasticplate}, the Kolmogorov equation \cite{lerousseaumoyano2016kolmogorov}, and parabolic equations involving quadratic (possibly non-self-adjoint) operators \cite{beauchardjamingpravda-starov2021spectral,martinpravda-starov2023spectral}. The explicit dependence in $\delta$ of the constant appearing in \eqref{eq: spec} provided by Kovrizhkin \cite{kovrizhkin2001relatedtoLog-Ser} also allows to prove controllability results for weakly dissipative parabolic equations, typically involving fractional Laplacians with exponents smaller than $\frac{1}{2}$ \cite{alphonsekoenig2022nullcont_weaklydissheat}. The thickness condition on $\omega$ required for controllability of weakly dissipative equations involves thickness at asymptotically small scales (see \cite[Definition 3]{alphonsekoenig2022nullcont_weaklydissheat}). In a recent article \cite{bombachtautenhahn2024logvinenkosereda}, Kovrizhkin's spectral estimate with explicit dependence of the constant is generalized to Banach-space valued functions, opening up possible applications to parabolic systems and parabolic equations with a parameter.

Propagation of smallness (typically Carleman inequality-based) and spectral inequalities have also proved key to controlling the $[0,\Lambda]$ frequency range of Schrödinger \cite{miller2004howviolentfastcontrols,susunyuan2023obs-1d-schrod} and wave \cite{burqjoly2016expdecaydampedwaves} equations for arbitrarily large thresholds $\Lambda$. Controllability in the high-frequency asymptotic regime is then provided by other methods, typically of semiclassical nature.

Since the original article by Lebeau and Robbiano \cite{lebeaurobbiano1995controlechaleur}, the theory of spectral estimates on manifolds has encountered numerous developments, with applications to both controllability results and spectral geometry. In \cite{jerisonlebeau1999nodal}, the authors provide a spectral estimate from open sensor sets on compact manifolds, both with and without boundary, and use it to obtain bounds on the $(d-1)$-Hausdorff measure of nodal sets of sums of eigenfunctions of the Laplacian in the spirit of Donnelly and Fefferman \cite{donnellyfefferman1988nodalRiemannian,donnellyfefferman1990nodalRiemannianboundary}. Like in \cite{lebeaurobbiano1995controlechaleur}, the proof of the spectral inequality is based on interior and boundary Carleman estimates. In recent years, the spectral estimates of \cite{lebeaurobbiano1995controlechaleur,jerisonlebeau1999nodal} have been generalized to the case where $\omega$ is merely a Borel subset of the interior of the compact manifold, on the sphere \cite{dickeveselic2022sphericallogvinenko} and on general compact manifolds with or without boundary \cite{burqmoyano2021spectralestimates}. In the latter, Burq and Moyano use Logunov and Malinnikova's propagation of smallness estimates for gradients of harmonic functions \cite[Theorem 5.1]{logunovmalinnikova2018qtttve_propag_smallness} as a replacement for previously Carleman-based propagation of smallness to derive controllability properties from Borel and zero-measure sets on compact Riemannian manifolds with $W^{2,\infty}$ metrics. The passage from the boundary-free case to the case of Dirichlet or Neumann boundary conditions is performed using a double-manifold reflection argument.

The present article is part of a program aiming at generalizing spectral estimate results from thick sets to non-compact geometric settings by adapting the strategy first introduced by Burq and Moyano. Miller \cite{miller2004nullcontrolheatunbounded,miller2005uniquecontinuation} was the first to tackle that problem using the Lebeau--Robbiano method and to provide a necessary condition for spectral estimates on non-compact manifolds with Ricci curvature bounded below. The next steps in that direction focused on the full space $\mathbb{R}^d$ endowed with analytic \cite{lebeaumoyano2019spectralinequalities} and bounded \cite{burqmoyano2023propagationheat} perturbations of the flat metric. The first instance of a result on non-compact, non-flat manifolds was provided by Rose and Tautenhahn \cite{rosetautenhahn2023unique}, using different, Brownian-motion and heat-kernel based methods, in a general geometric framework (Ricci curvature bounded from below, namely the assumption of Theorem \ref{theo: main result, necessary part}). Their result holds for frequency ranges $[0,\Lambda]$ below a frequency threshold constrained by the curvature lower bound, thus not allowing to control the heat equation. The first high-frequency result for spectral estimates from thick (Borel) sensor sets on a non-compact, non-flat Riemannian manifold was provided by the second author in the specific case of the hyperbolic half-plane $\mathbb{H}$ \cite{rouveyrol2023spectralH2}. In that article, the spectral estimate is deduced from \eqref{eq: thickness} by adapting the approach of Burq and Moyano \cite{burqmoyano2023propagationheat,burqmoyano2021spectralestimates} on a suitably defined covering of the manifold, while the necessary character of the thickness condition is shown using heat kernels in the spirit of \cite{wangwangzhangzhang2019spectralineq_heat}. The present article can be seen as a continuation of \cite{rouveyrol2023spectralH2}. The main contributions made here and absent from \cite{rouveyrol2023spectralH2} are the following: \begin{enumerate}
    \item in Section \ref{sec: funnel}, we tackle the fact that the denominator $\vol(B_z(R))$ (with $R$ fixed) can go to zero on a manifold with null global injectivity radius (in our case, when $z$ tends to the end of a cusp). In $\mathbb{H}$, that denominator was constant with respect to $z$ and included in the $\delta(R)$ parameter. We hope that the periodic lifting argument provided may be generalized to a larger class of manifolds, relaxing the constant-curvature assumption.
    \item in Lemma \ref{lem:prop_smallness_compact}, we prove that controllability of the heat equation on a manifold with ends only requires a sensor set $\omega$ active in the ends. The idea of that section is to use Carleman-based propagation of smallness estimates to propagate smallness from any end to the compact core.
    \item in Section \ref{sec: proof necessary condition}, we generalize the idea of combining observability estimates and heat kernel bounds on manifolds developed in \cite[Section 2.3]{rouveyrol2023spectralH2} to manifolds with smooth metric and Ricci curvature bounded from below. Again, attention must be paid to recovering the small denominator $\vol(B_z(R))$ which could be ignored on $\mathbb{H}$. 
\end{enumerate}

We now mention two questions  which we deem interesting but chose not to address in the article. The first one concerns the dependence of the constant $Ce^{C\Lambda}$ in \eqref{eq: spec} on the geometry and on the thickness parameters $(R, \delta)$, which has applications of its own as discussed earlier. On compact manifolds, this (more specifically, the dependence of the observability cost for the heat equation) was studied by Miller \cite{miller2004geometricbounds}. The results were refined by Laurent and Léautaud \cite{laurentleautaud2021conjlucmiller} in the case where the sensor set $\omega$ is a ball of decreasing radius. Concerning the behavior of observability and control costs for the heat equation as a function of the constants appearing in the spectral estimate, we refer to \cite{wangwangzhangzhang2019spectralineq_heat,nakictaufertautenhahnveselic2020homogenizationcontrolcostheat} and the references therein. Lastly, Theorem 5.1 of \cite{rosetautenhahn2023unique} gives an explicit dependence of the spectral estimate constant in terms of the Ricci curvature lower bound and the thickness parameters for frequencies smaller than a specific (and again, explicit) threshold. Note that the thickness parameters used in \cite{rosetautenhahn2023unique} are not exactly those of the present article: the $(R,\rho)$-thickness condition used in \cite{rosetautenhahn2023unique}, under which the sensor set $\omega$ must contain a ball of radius $\rho$ inside every ball of radius $R$, is similar to \eqref{eq: thickness} in spirit but strictly implies it. 

The other main question that we chose not to tackle in the present
article is that of sensitivity of the thickness condition to
potentials. Understanding spectral estimates for the operator $-\Delta
+ V$ opens up applications to controllability of the heat equation
with a potential (see for example
\cite{beauchardjamingpravda-starov2021spectral,martinpravda-starov2023spectral}
in the harmonic case) and random Schrödinger operators
\cite{rojas-molinaveselic2013uniquecont_rdmschop,klein2013UC_rdmschr_wegnerestimates}. On
the full space $\mathbb{R}^d$,
\cite{lebalchmartin2024quantitative} deals with spectral estimates from sets
satisfying \eqref{eq: thickness} for $-{\rm div}(A(x)\nabla u) + V$, with $V$ an $L^\infty$ potential and $A$ uniformly elliptic with Lipschitz entries. Like that of Theorem \ref{theo: main result, sufficient part}, their proof is based on the Logunov--Malinnikova propagation of smallness estimate for gradients. Previous results in the case of bounded potentials used an analytic framework \cite{lebeaumoyano2019spectralinequalities} or a decay assumption at infinity \cite[Theorem 2]{zhu2024spineq_schvarpot}. In the case of unbounded, polynomially growing potentials of the full space, it has been shown that spectral estimates hold from sets of decaying density \cite{alphonseseelmann2023quantitativespectralinequalitiesanisotropic,dickeseelmannveselic2024spectralineq,wang(yunlei)2024quantitativepropsmallness,zhu2024spineq_schvarpot}. The idea in that case is to combine propagation of smallness arguments with frequency-quantitative decay (Agmon) estimates for eigenfunctions \cite{agmon1982book,alphonse2024null-controllabilityShubinAgmon}. We hope to study Agmon estimates on non-compact manifolds and decaying thickness conditions in presence of a potential in forthcoming work. 

\subsubsection{About propagation of smallness inequalities and their application on manifolds}
\label{sec: bibliography propagation of smallness}

We conclude this bibliographical overview by covering the propagation of smallness tools used to prove Theorem \ref{theo: main result, sufficient part}. Precise statements of the corresponding theorems can be found in Appendix \ref{sec: appendix propagation of smallness results}. Because the proof of Theorem \ref{theo: main result, necessary part} is of a different, geometric-analytic nature, discussion of the tools used to prove that result is postponed to Section \ref{sec: proof necessary condition}.

The main propagation of smallness tool used in the article is the result by Logunov and Malinnikova \cite[Theorem 5.1]{logunovmalinnikova2018qtttve_propag_smallness} for gradients of harmonic functions. The proof is a combination of estimates for the critical sets of solutions of elliptic equations \cite{cheegernabervaltorta2015criticalsetselliptic}, properties of the doubling index of these solutions and combinatorial lemmas, which were used by the same authors to obtain results on nodal set bounds \cite{logunovmalinnikova2018nodaldim23,logunov2018nodalpolynomial,logunov2018nodalnadirashvili}. The estimate holds on subdomains of the Euclidean space and the constant involved depends on the Lipschitz and ellipticity constants of the metric as well as the Hausdorff measure $\delta$ of the set from which smallness is propagated. For these reasons, we need to apply the Logunov--Malinnikova estimate in charts satisfying at least the following properties:
\begin{enumerate}
    \item the charts must contain a ball of radius $R$ so that the thickness assumption \eqref{eq: thickness} may allow to recover a uniform mass of the sensor set $\omega$ inside the chart;
    \item the atlas is locally uniformly finite, meaning each point of the manifold lies at the intersection of a uniformly bounded number of charts;
    \item the domain of the charts must be uniform, and the metric must be controlled in terms of ellipticity and Lipchitz constants in the corresponding coordinate system.
\end{enumerate}

In the way it is used in the present article, the Logunov--Malinnikova estimate can be thought of as a replacement for Carleman-estimate based propagation of smallness (see for example \cite[Lemmas 14.4, 14.5]{jerisonlebeau1999nodal}). One major improvement in the former is that it requires only very mild assumptions on the set from which smallness propagates (positivity of some $(d-1-\epsilon)$ Hausdorff content), while the Carleman method requires open sets -- or open sets of the boundary, allowing for propagation of smallness from a $(d-1)$-dimensional set. The usual proof of Carleman estimates relies on pseudodifferential calculus with a large parameter and a commutator estimate based on the G\aa rding inequality. This gives a weighted unique continuation estimate (the Carleman inequality), and a good choice of the weight function allows to deduce propagation of smallness results (see \cite[Section 3]{lebeaurobbiano1995controlechaleur} or \cite[Sections 14.3, 14.4]{jerisonlebeau1999nodal}). A comprehensive overview of Carleman estimates and their applications to control problems is the book of Le Rousseau, Lebeau and Robbiano \cite{lerousseaulebeaurobbiano2022carleman1,lerousseaulebeaurobbiano2022carleman2}. For a shorter introduction (in French), we refer the reader to \cite{lebeau2015carlemanSMF}.

\subsection{Strategy of proof}

\begin{figure}[b]
    \centering
    \begin{tikzpicture}[scale=2.7]
        \draw[domain=-2.8:2.8, smooth, variable=\x, black, very thick] plot ({\x}, {0.8*exp(-\x*\x/2)}); 
        \draw[domain=-2.8:2.8, smooth, variable=\x, black, very thick] plot ({\x}, {-0.8*exp(-\x*\x/2)}); 
        \draw[domain=-0.5:0.5, smooth, variable=\x, black, very thick] plot ({\x}, {0.5*(\x-0.4)*(\x+0.4)}); 
        \draw[domain=-0.4:0.4, smooth, variable=\x, black, very thick] plot ({\x}, {-0.5*(\x-0.4)*(\x+0.4)}); 
        
        \draw[RoyalBlue, thick] (-1.5,0.) ellipse (0.06 and 0.26);
        \draw[RoyalBlue, thick] (1.5,0.) ellipse (0.06 and 0.26);
        \draw[RoyalPurple, thick, dashed] (-1.3,0.) ellipse (0.06 and 0.34);
        \draw[RoyalPurple, thick, dashed] (1.3,0.) ellipse (0.06 and 0.34);

        \draw[RoyalPurple, thick, dashed] (-1.3,-0.8) -- (-1.3,-0.9) -- (1.3,-0.9) -- (1.3,-0.8);
        \draw[RoyalPurple, thick, dashed] (0.,-0.9) node[below] {$K$};
        \draw[RoyalPurple, thick, dashed] (-2.8,-0.8) -- (-2.8,-0.9) -- (-1.3,-0.9) -- (-1.3,-0.8);
        \draw[RoyalPurple, thick, dashed] (-2.05,-0.9) node[above] {$\mathcal{C}_1$};
        \draw[RoyalPurple, thick, dashed] (2.8,-0.8) -- (2.8,-0.9) -- (1.3,-0.9) -- (1.3,-0.8);
        \draw[RoyalPurple, thick, dashed] (2.05,-0.9) node[above] {$\mathcal{C}_2$};

        \draw[RoyalBlue, thick] (-1.7,0.95) -- (-1.7,1.05) -- (1.7,1.05) -- (1.7,0.95);
        \draw[RoyalBlue, thick] (0.,1.05) node[above] {$\Tilde{K}$};
        \draw[RoyalBlue, thick] (-2.8,0.8) -- (-2.8,0.9) -- (-1.5,0.9) -- (-1.5,0.8);
        \draw[RoyalBlue, thick] (-2.15,0.9) node[below] {$\Tilde{\mathcal{C}}_1$};
        \draw[RoyalBlue, thick] (2.8,0.8) -- (2.8,0.9) -- (1.5,0.9) -- (1.5,0.8);
        \draw[RoyalBlue, thick] (2.15,0.9) node[below] {$\Tilde{\mathcal{C}}_2$};
    \end{tikzpicture}
    \caption{Sketch of a hyperbolic surface of genus 1 with two cusps. In purple dashed lines, the decomposition into its compact core and cusps M = $K \sqcup \mathcal{C}_1 \sqcup \mathcal{C}_2$. In blue full lines, the regions $\Tilde{\mathcal{C}}_1$, $\Tilde{\mathcal{C}}_2$ where we prove propagation of smallness estimates, and the compact region $\Tilde{K}$ overlapping $\Tilde{\mathcal{C}}_1$ and $\Tilde{\mathcal{C}}_2$.}
    \label{fig: example hyperbolic surface}
\end{figure}
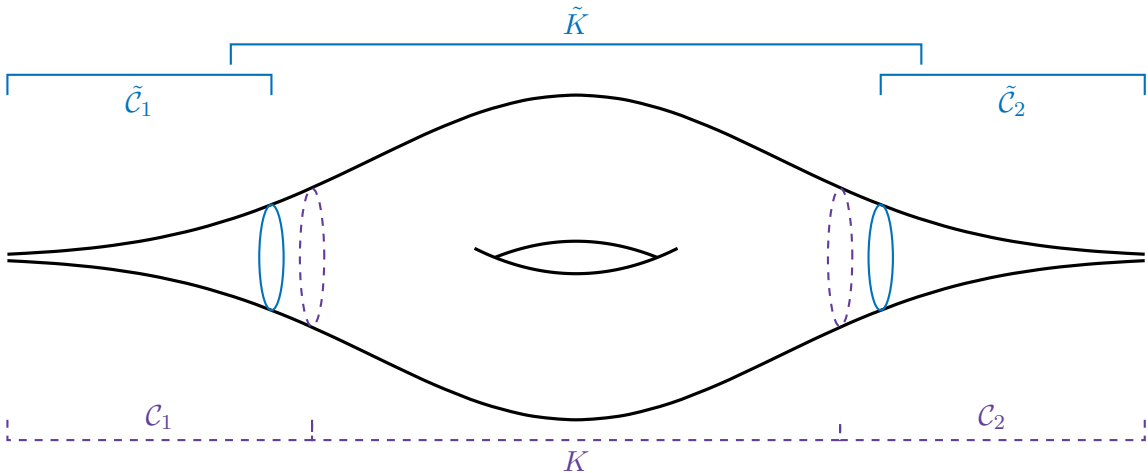

The proof of the sufficient condition consists in applying propagation of smallness inequalities in charts, to a harmonic function which contains all relevant information on $\Pi_\Lambda u$. To that end, we add an auxiliary variable $t$ and introduce the harmonic extension
 \begin{equation}
    \label{eq: harmonic extension}
    v_\Lambda(t, \cdot) = \int_{\lambda = 0}^\Lambda \frac{\sinh(\lambda t)}{\lambda} dm_\lambda u,
\end{equation} which satisfies \begin{equation}
    \label{eq: harmonic extension properties}
    v_\Lambda(t=0, \cdot) = 0, \quad \partial_t v_\Lambda(t=0, \cdot) = \Pi_\Lambda u, \quad \nabla_g v_\Lambda (t=0, \cdot) = 0, \quad \partial_t^2 + \Delta_g v_\Lambda = 0.
\end{equation}

The proof of Theorem \ref{theo: main result, sufficient part} adapts Burq and Moyano's strategy for Lipschitz metrics on the full space $\mathbb{R}^d$ \cite{burqmoyano2021spectralestimates} (itself inspired by the pioneering work of Jerison and Lebeau \cite{jerisonlebeau1999nodal}). The general idea is to apply propagation of smallness results for the harmonic extension $v_\Lambda$ in local charts and to derive \eqref{eq: spec} from summing these estimates together. We now describe the steps of the proof of Theorem \ref{theo: main result, sufficient part} and how they differ from \cite{burqmoyano2023propagationheat,burqmoyano2021spectralestimates} and the second author's previous work \cite{rouveyrol2023spectralH2}.

The first step is to apply Logunov and Malinnikova's propagation of
smallness inequality
\cite{logunovmalinnikova2018qtttve_propag_smallness} for gradients in
a suitably chosen ball of the hyperbolic half-plane $\mathbb{H}$. We refer to section \ref{sec: bibliography propagation of smallness} for a quick bibliographical overview of this result. This is essentially the only step where the explicit expression of the metric is used, as it allows to construct charts appropriate for applying the Logunov--Malinnikova estimate with uniform constants. The inequality obtained (Lemma \ref{lemma: propagation of smallness ball}) and its proof are close to those used in \cite{rouveyrol2023spectralH2}.

We then decompose the surface with hyperbolic ends into its cusps $\mathcal{C}_j$, funnels $F_j$, and a compact core $K$. The propagation of smallness inequality for balls in $\mathbb{H}$ is extended to funnels and cusps using periodic lifting arguments based on their respective structure (see Figure \ref{fig: funnel and cusp in H2}). Specifically, the surface is decomposed as depicted in Figure \ref{fig: example hyperbolic surface}: we prove propagation of smallness from the intersection of $\omega$ with slightly smaller subsets $\Tilde{F}_j$ , $\Tilde{\mathcal{C}}_j$ of the funnels and cusps, respectively.

We then denote $\Tilde{K}$ an open, relatively compact enlargement of $K$ that intersects every $\Tilde{F}$, $\Tilde{\mathcal{C}}$. In that relatively compact set, we use Carleman-based propagation of smallness results from \cite{lerousseaulebeaurobbiano2022carleman2} to reduce the question of the compact core to propagation of smallness inequalities in the ends. We finish by patching together estimates in every tilded subset to obtain the spectral estimate.

The proof of Theorem \ref{theo: main result, necessary part} is
inspired by a similar result from
\cite{wangwangzhangzhang2019spectralineq_heat} on the flat
space. Indeed, the spectral estimate implies an observability
inequality for solutions of the free heat equation as stated in Lemma
\ref{lemma: observability}. The thickness inequality can then be
recovered by testing this observability inequality for short times
with a well-chosen heat kernel. Key elements of the proof are heat
kernel bounds and volume doubling estimates that hold on manifolds
with smooth metrics and Ricci curvature bounded from below. As
remarked earlier, we hope that the regularity assumption on the metric
$g$ can be weakened, but we believe that a lower bound on Ricci
  curvature is a crucial condition in Theorem \ref{theo: main result, necessary part}. Indeed, there are examples of ill-behaved heat kernels when that assumption is dropped.

The plan of the article is the following. In Section \ref{sec: proof sufficient condition}, we prove Theorem \ref{theo: main result, sufficient part}. The proof is divided into three subsections. In the
first, we prove a propagation of smallness result for ``basic
pieces'': balls in hyperbolic space (with uniformity with respect to the controlled
set $\omega$) and on the compact core (without uniformity, but we are
dealing with a \emph{fixed} hyperbolic surface in this article). In subsection \ref{sec: funnel}, we show how this result extends to funnels and cusps using periodic lifting arguments. We finish Section \ref{sec: proof sufficient condition} by showing propagation of smallness from an end to the compact core and patching estimates on funnels, cusps and the compact core together to deduce the spectral estimate. Section \ref{sec: proof necessary condition} is dedicated to the proof of Theorem \ref{theo: main result, necessary part}. The Appendix contains statements and comments on the propagation of smallness inequalities used throughout Section \ref{sec: proof sufficient condition}.

\subsection*{Notation}

Throughout the rest of the article, we denote by $z$ a generic element in the manifold $M$, $d_g(z,z')$ the geodesic distance between two points $z, z' \in M$ and $d\vol_g$ the Riemannian measure corresponding to the metric $g$. $B_z(R)$ denotes a geodesic ball of center $z$ and radius $R$.

The symbol $\lesssim$ means that the left-hand-side is bounded by the right-hand-side up to an implicit multiplicative constant which is independent of relevant quantities. The dependence of implicit constants is specified as an index when deemed necessary. $A \simeq B$ means that both $A \lesssim B$ and $B \lesssim A$ hold.

\subsection*{Acknowledgements}

M.R. would like to thank Ruoyu P.T. Wang for pointing out the interest of studying spectral estimates on surfaces with cusps, and Nicolas Burq and Luc Robbiano for interesting discussions about the topic of this article. M.R. acknowledges funding from a CDSN PhD grant of ENS Paris-Saclay, via the Hadamard Doctoral School of Mathematics.
A.D. acknowledges funding by the ERC-2022-ADG grant GEOEDP (project 101097171).

\section{Thickness of $\omega$ implies the spectral estimate}
\label{sec: proof sufficient condition}

\subsection{Local estimates for propagation of smallness}
\label{sec: propagation of smallness}

As explained in the introduction, we aim to obtain propagation of
smallness estimates on both funnels and cusps. These estimates will in
turn derive from an inequality localized on balls of the
hyperbolic half-plane, the proof of which is borrowed from
\cite{rouveyrol2023spectralH2}. These estimates will be applied with
$w$ and $\mathfrak{u}$ equal to appropriate periodic liftings to
$\mathbb{H}$ of $v_\Lambda$ and $\Pi_\Lambda u$, respectively. The
statement is made precise by the following lemma.

Unless explicitly noted, in this article, $L^2$ norms of functions on
Riemannian manifolds and gradients, are considered with respect to the
Riemannian structures.

For example, given $V\subset
\mathbb{H}$ open and
$u:V\to \mathbb{C}$, we denote
\[
  \|u\|_{L^2(V)}=\int_V |u(x,y)|^2\frac{{\rm d}x{\rm d}y}{y^2}
\]
and
\[
  \nabla u:(x,y)\mapsto (y^{-1}\partial_xu, y^{-1}\partial_yu).
\]
We recall that, in the half-plane model $\mathbb{H}$ of the hyperbolic
space, the geodesic ball $B_{\mathbb{H}}((x,y),R)$ centred at a point
$(x,y>0)$ and of radius $R$ coincides with the Euclidean ball of
centre $(x,y\cosh(R))$ and radius $\sinh(R))$.

The following local lemma will be used to control the non-compact,
hyperbolic ends of the manifold. It consists in a propagation of
smallness estimate, on hyperbolic balls of fixed radius, where the
constants do not depend on the geometry of the controlled set but only
on its volume.
\begin{lemma}
    \label{lemma: propagation of smallness ball}
    Let $R,\eta,T,\delta$ be positive constants. There exists $C>0$ and
    $\alpha\in (0,1)$ such that the following is true.

    Let $z\in \mathcal{H}$ and let $\mathfrak{u} : B(z,e^{\eta}R) \rightarrow
    \mathbb{R}$ and $w : \mathbb{R}_t \times B(z,e^{\eta}R)
    \rightarrow \mathbb{R}$ satisfy $$(\partial_t^2 +
    \Delta_\mathbb{H}) w = 0, \quad \partial_t w(t=0) = \mathfrak{u},
    \quad \nabla_\mathbb{H} w(t=0) = 0.$$ Consider a subset
    $\hat{\omega} \subset B(z,R)$ satisfying
    $\vol_\mathbb{H}(\widehat{\omega}) \geq \delta$.

    Then
    \begin{equation}
        \label{eq: propagation of smallness, ball}
        \|\mathfrak{u}\|_{L^2(B(z,R))}+\|\partial_t
        w\|_{L^2((-\frac{T}{2},\frac{T}{2}) \times B(z,R))}\leq C \|\mathfrak{u} \|_{L^2(\omega)}^{\alpha}\| w\|_{H^3((-T,T) \times B(z,e^{\eta}R))}^{1-\alpha}.
      \end{equation}
\end{lemma}
\begin{proof}
In this proof, we will use both the intrinsic Riemannian structure on
$\mathbb{H}$, as well as the one coming from the Euclidean structure on the
plane. To distinguish between the two, we will use the subscripts
$\mathbb{H}$ and $\mathbb{R}^2$ on $L^2$ spaces, gradients, and balls.

We first apply an isometry $\phi$ of $\mathbb{H}$ which sends $B(z,R)$ to
$B(i,R)$; by design, $\phi$ preserves $\Delta_{\mathbb{H}}$ and $L^2_{\mathbb{H}}$ norms.

Accordingly, we define $W(t, X, Y) = w(t,
\phi^{-1}(X,Y))$, which satisfies $$(\partial_t^2 + Y^2
\Delta_{\mathbb{R}^2}) W(t, X, Y) = (\partial_t^2 +
\Delta_\mathbb{H}) w(t, x, y) = 0, \quad \forall (t,X, Y) \in \mathbb{R}\times
B_{\mathbb{H}}(i,R).$$

Multiplying by $\frac{1}{Y^2}$, we get $$(\partial_t \frac{1}{Y^2}
\partial_t + \Delta_{\mathbb{R}^2})W = 0 \text{ over }
\mathbb{R}\times B_{\mathbb{H}}(i,R),$$ which in turn can be written in the form
\[
  \Div(A
\nabla_{t,X,Y} W) = 0 \qquad \qquad A(t,X,Y) = \begin{pmatrix}
    \frac{1}{Y^2} & 0 & 0 \\
    0 & 1 & 0 \\
    0 & 0 & 1
\end{pmatrix}.\] For any $T > 0$, $A$ is uniformly elliptic with
Lipschitz coefficients over the domain \[(-T,
  T)~\times~B_{\mathbb{H}}(i,R)~=~(-T, T)_t~\times~B_{\mathbb{R}^2}(i\cosh(R),\sinh(R)),\] and its ellipticity and
Lipschitz constants are independent of $T$ and $z$: in the sense
of quadratic forms one has
\[
  e^{-2R}\leq A\leq e^{2R}
\]
and (say in operator norm)
\[
  \|A(z_1)-A(z_2)\|\leq e^{3R}|z_1-z_2|_{\infty}.
\]

Fix $T_1=\frac T2>0$ and $T_2>T_1$. We denote
\begin{equation}
    \label{eq: logunov-malinnikova sets (t, X, Y)}
    \left\{\begin{array}{ll}\mathcal{K} = (-T_1, T_1) \times
             B_{\mathbb{H}}(i,R)\\ \Omega = (-T_2, T_2) \times
             B_{\mathbb{H}}(i,e^{\eta}R)\\ F =
             \phi(\widehat{\omega})\\ E = \{0\} \times
             F.
           \end{array}\right.
         \end{equation}

  Then by the assumption on $\widehat{\omega}$, $$\int_{\widehat{\omega}} \frac{dxdy}{y^2} = \int_{\phi(\widehat{\omega})} \frac{dXdY}{Y^2} \geq \delta.$$ Thus, \begin{equation}
\label{eq: lower bound after scaling}
\vol_{\mathbb{R}^3}(\phi(\widehat{\omega})) \geq \frac{\delta}{e^{2R}}=:\tilde{\delta}.\end{equation} By Logunov and Malinnikova's theorem \cite[Theorem 5.1]{logunovmalinnikova2018qtttve_propag_smallness}, the statement of which is recalled in Appendix \ref{sec: appendix propagation of smallness results}, we get \begin{equation}
    \label{eq: logunov malinnikova L infty bound}
    \sup_{\mathcal{K}} |\nabla_{\mathbb{R}^3} W| \leq C \sup_{E} |\nabla_{\mathbb{R}^3} W|^\alpha \sup_{\Omega} |\nabla_{\mathbb{R}^3} W|^{1-\alpha}.
\end{equation}  Here, the constants $C > 0$ and $\alpha\in (0,1)$ depend on
$\delta,R,\eta,T_1,T_2$, but do not depend on $\mathfrak{u}$ and $z$. 

Our next goal is to transform inequality \eqref{eq: logunov
  malinnikova L infty bound} into an
inequality of the same form (with different constants) involving
$L^2$ norms instead of $\sup$ norms. To this end we adapt the strategy used in
\cite[Section 2]{burqmoyano2021spectralestimates}.

First of all, we claim that either $\mathfrak{u}=0$ on $\mathcal{B}_{\mathbb{H}}(z,R)$, or
\[
  {\rm vol}(\{(X,Y)\in
  F,\nabla_{\mathbb{R}^3}W(t=0,X,Y)=0\})=0.
\]
Indeed, letting $Z=\{(X,Y)\in
  F,\nabla_{\mathbb{R}^3}W(t=0,X,Y)=0\}$, if the
  volume of $Z$ is non-empty, then the Logunov-Mallinikova estimate
  yields, for some $\alpha(Z)\in (0,1)$,
  \[
    \sup_{\mathcal{K}} |\nabla_{\mathbb{R}^3} W| \leq C(Z)
    \sup_{Z} |\nabla_{\mathbb{R}^3} W|^{\alpha(Z)} \sup_{\Omega}
    |\nabla_{\mathbb{R}^3} W|^{1-\alpha(Z)}=0.
  \]
  In turn, given the definition of $w$ and $W$, if
  $W$ vanishes on $K$ then $\mathfrak{u}$ vanishes on
  $B_{\mathbb{H}}(z,R)$, and then the initial claim is immediate. From
  now on we suppose that $\mathfrak{u}$ is not identically zero on $\mathcal{B}_{\mathbb{H}}(z,R)$.

  Let now $\epsilon>0$
small enough. Let $$a = \left( \epsilon \frac{\sup_{\mathcal{K}}
    |\nabla_{\mathbb{R}^3} W|}{\sup_{\Omega} |\nabla_{\mathbb{R}^3}
    W|^{1-\alpha}} \right)^\frac{1}{\alpha}, \quad \text{where}\quad F' =
\{(X,Y) \in F, |\nabla_{\mathbb{R}^3} W(t=0, X, Y)| \leq a\}.$$
Let $E' = \{0\} \times F'$. For $\epsilon>0$ small
enough one has $\vol_{\mathbb{R}^3}(F') \leq \frac{\delta}{2}$, since
$F'$ increases with $\epsilon$ and the intersection over
$\epsilon\in (0,1)$ of the $F'(\epsilon)$ is $Z$. 
Thus \begin{equation} \label{eq: L2 bound from below}\begin{aligned}
    \int_{F} |\nabla_{\mathbb{R}^3} W(t=0, X, Y)|^2 dX
    dY &\geq \int_{F \setminus F'} |\nabla W(t=0,
    X, Y)|^2 dXdY \\ &\geq a^2 \frac{\delta}{2}\\ &= \frac{\delta}{2}
    \left( \epsilon \frac{\sup_{\mathcal{K}} |\nabla_{\mathbb{R}^3}
        W|}{\sup_{\Omega} |\nabla_{\mathbb{R}^3}
        W|^{1-\alpha}} \right)^\frac{2}{\alpha}.
  \end{aligned}
\end{equation}

Now, by definition $\nabla_{\mathbb{R}^3} W(t=0) = (\mathfrak{u} \circ
\phi^{-1},0,0)$ so that,
denoting $U = \mathfrak{u} \circ \phi^{-1}$, we get 
\begin{equation}
\label{eq: burq--moyano trick}
\begin{aligned}\int_{\mathcal{R}_1} |U(X,Y)|^2 dXdY &\leq {\rm vol}_{\mathbb{R}^3}(\mathcal{R}_1) \sup_{\mathcal{K}} |\nabla_{\mathbb{R}^3} W|^2 \\
&\lesssim \left( \int_{F} |\nabla_{\mathbb{R}^3} W(t=0, X, Y)|^2 dX dY\right)^\alpha \sup_{\Omega}|\nabla_{\mathbb{R}^3} W|^{2-2\alpha} \\
& \lesssim \left( \int_{F} |U(X, Y)|^2 dX dY\right)^\alpha \sup_{\Omega}|\nabla_{\mathbb{R}^3} W|^{2-2\alpha}. \end{aligned}\end{equation}
The implicit multiplicative constant in the last inequality depends
only on $\delta, \epsilon, T_1, T_2, R, \eta$, but does not depend on $z$ and $\mathfrak{u}$.

Since the change of coordinates $\phi$ preserves the Riemannian structure on $\mathbb{H}$,
 for every $f\in L^2(B_{\mathbb{H}}(i,R)))$,
\[
  e^{-R}\|f\|_{L^2_{\mathbb{R}^2}(B_{\mathbb{H}}(i,R))}\leq
  \|f\|_{L^2_{\mathbb{H}}(B_{\mathbb{H}}(i,R))}=\|f\circ
  \phi\|_{L^2_{\mathbb{H}}(B_{\mathbb{H}}(z,R)))}\leq
  e^{R}\|f\|_{L^2_{\mathbb{R}^2}(B_{\mathbb{H}}(i,R)))}.
  \]
Applying this to $f = U$ and $f = U
\mathbbm{1}_{\phi(\widehat{\omega})}$, we get  respectively
\begin{equation} \label{eq: density appears naturally}
  \|\mathfrak{u}\|_{L^2_\mathbb{H}(B_{\mathbb{H}}(z,R))}\leq
  e^{R}\|U \|_{L^2_{\mathbb{R}^2}(B_{\mathbb{H}}(i,R))}, \quad
  \|U\|_{L^2(F)} \leq e^{R}
  \| \mathfrak{u}\|_{L^2_\mathbb{H}(\widehat{\omega})}.\end{equation}
Thus, from \eqref{eq: burq--moyano trick} we obtain  \begin{equation}
    \label{eq: logunov malinnikova L2} \|\mathfrak{u}\|_{L^2_\mathbb{H}(B_{\mathbb{H}}(z,R))}^2 \lesssim \|\mathfrak{u}\|_{L^2_\mathbb{H}(\widehat{\omega})}^{2\alpha} \sup_{\Omega}|\nabla_{\mathbb{R}^3} W|^{2-2\alpha}.
\end{equation}

To conclude, we now adapt the Sobolev injection process of \cite[Section
2.2]{rouveyrol2023spectralH2}, itself inspired by
\cite{burqmoyano2023propagationheat}, to bound
$\sup_{\Omega}|\nabla_{\mathbb{R}^3} W|^2$. Recall that $\Omega
= (-T_2, T_2) \times \mathcal{B}_{\mathbb{H}}(i,e^{\eta}R)$. Let $T_3>T_2$ and let $\Omega^{(1)} = (-T_3,T_3) \times
\mathcal{B}_{\mathbb{H}}(i,e^{2\eta}R)$. Let $\chi\in C^{\infty}_c(\Omega^{(1)},\mathbb{R})$ be
equal to 1 on $\Omega$. By the Sobolev injection in
dimension 3, we have $$\sup_\Omega | \nabla_{\mathbb{R}^3}
W|^2 \leq \|\chi \nabla_{\mathbb{R}^3} W
\|_{H^2_{\mathbb{R}^3}(\Omega^{(1)})}^2.$$ Now, by elliptic regularity (see \cite[Section
6.3]{evans2022book}) applied to the operator $-(\partial_t^2 + Y^2
\Delta_{\mathbb{R}^2})$, which is uniformly elliptic over $\Omega^{(1)}$ as
remarked earlier, we obtain \begin{equation}
\label{eq: elliptic regularity}
\|\chi\nabla_{\mathbb{R}^3} W \|_{H^2_{\mathbb{R}^3}(\Omega^{(1)})}^2 \lesssim \|(1 - \partial_t^2 - Y^2 \Delta_{\mathbb{R}^2}) \chi\nabla_{\mathbb{R}^3}W \|_{L^2_{\mathbb{R}^3}(\Omega^{(1)})}^2.
\end{equation}
Since the weight $Y$ and its inverse are bounded on $\Omega^{(1)}$
along with their successive derivatives, we obtain
\begin{align*}
  \|(1 - \partial_t^2 - Y^2 \Delta_{\mathbb{R}^2})
  \chi\nabla_{\mathbb{R}^3}W \|_{L^2_{\mathbb{R}^3}(\Omega^{(1)})}^2&\lesssim
                                                              \sum_{p+q=1}^3\|\partial_t^p\nabla_{\mathbb{R}^2}^qW\|_{L^2_{\mathbb{R}^3}(\Omega^{(1)})}^2\\
  &\lesssim
    \sum_{p+q=1}^3\|\partial_t^p\nabla_{\mathbb{H}}^qW\|^2_{L^2_{\mathbb{R}\times
    \mathbb{H}}(\Omega^{(1)})}.
\end{align*}
Together with \eqref{eq: logunov malinnikova L2}, this concludes the
proof of the control of
$\|\mathfrak{u}\|^2_{L^2(\mathcal{R}_{\kappa})}$ in \eqref{eq:
  propagation of smallness, ball}.

To bound $\|\partial_tw\|^2_{L^2_{\mathbb{R}\times\mathbb{H}}((-\frac T2,\frac
  T2)\times
  B_{\mathbb{H}}(z,R))}=\|\partial_tW\|^2_{L^2_{\mathbb{R}\times
    \mathbb{H}}(\mathcal{K})}$, we replace the first line of
  \eqref{eq: burq--moyano trick} with the volume inequality
  \[
    \|\partial_tW\|^2_{L^2_{\mathbb{R}\times
        \mathbb{H}}(\mathcal{K})}\lesssim {\rm
      vol}(B(i,R))\sup_{\mathcal{K}}|\nabla_{\mathbb{R}^3}W|^2
  \]
  and continue from there. This concludes the proof.
\end{proof}

Another variant of ``propagation of smallness'' will be used on the
compact core of the manifold; here we are free to relax the
hyperbolicity assumption on the metric -- at the cost of a dependence
of the geometry on the constants.

\begin{lemma}\label{lem:prop_smallness_compact}
  Let $N$ be a connected, compact, smooth Riemannian surface with
  non-empty boundary $\partial N$. Let $V$ be an open neighbourhood of
  $\partial N$. Let $T>0$. There exist $C>0$ and $\alpha\in (0,1)$ such that, for any
  $\mathfrak{u}:N\to \mathbb{C}$ and any $w:\mathbb{R}\times N\to
  \mathbb{C}$ satisfying
  \[
    (w(0,x),\partial_tw(0,x))=(0,\mathfrak{u}(x))\qquad \qquad
    (\partial_t^2+\Delta)w=0\]
  one has 
  \begin{equation}
    \label{eq: propagation of smallness from boundary}
    \|\mathfrak{u}\|_{L^2(N\setminus V)} +\|\partial_t
    w\|_{L^2((-\frac{T}{2},\frac{T}{2}) \times (N\setminus V))}\\
    \leq C \|\mathfrak{u} \|_{L^2(V)}^{\alpha}\|w\|_{H^2((-T,T)\times
      N)}^{1-\alpha}. 
  \end{equation}
\end{lemma}
\begin{proof}
  Let $V_0=V$; introduce a family $V_3\subset \subset V_2 \subset \subset V_1\subset
  \subset V_0\subset \subset V_{-1}\subset \subset V_{-2}$ of open
  neighbourhoods of $\partial N$; without loss of generality
  $N\setminus V_j$ is connected. 
  On the manifold $\mathcal{M}=[-T,T]\times N$, consider
  \begin{align*}
    \mathcal{U}&=(-\tfrac T2, \tfrac T2)\times (N\setminus V_2)\\
    \mathcal{V}&=(-T,T)\times N\\
    \mathcal{X}&=(-\tfrac T4, \tfrac T4)\times (V_0\setminus V_1).
  \end{align*}
  
  We are in position to apply Theorem \ref{theo: Carleman-based
    propagation of smallness} to $\nabla_{\mathcal{M}} w$, which is
  also harmonic for $\partial_t^2+\Delta_N$. This yields, for some $\beta\in (0,1)$,
  \[
    \|w\|_{H^2(\mathcal{U})}\leq
    C\|w\|_{H^1(\mathcal{X})}^{\beta}\|w\|_{H^2(\mathcal{V})}^{1-\beta}.
  \]
  Here, the left-hand side controls $\|\partial_t
    w\|^2_{L^2((-\frac{T}{2},\frac{T}{2}) \times (N\setminus V))}$
    since $H^1\subset H^2$ and $N\setminus V\subset N\setminus V_2$; it also controls
    $\|\mathfrak{u}\|_{L^2(N\setminus V)}$ by continuity of the trace.

    It remains to control $\|w\|^{\beta}_{H^1(\mathcal{X})}$ and to
    this end we apply Theorem \ref{theo: Carleman boundary} to $w$ on
    the connected components of 
    \begin{align*}
      \mathcal{M}'&=\pm [0,\tfrac T2]\times (V_{-2}\setminus V_3)\\
      \mathcal{U}'&=\pm [0,\tfrac T4)\times (V_0\setminus V_1)\\
      \mathcal{V}'&=\pm [0,\tfrac T3)\times (V_{-1}\setminus V_2)\\
      \mathcal{X}'&=\{0\}\times (V_0\setminus V_1).
    \end{align*}
    By summing the estimates on positive and negative times, we obtain
    precisely, for some $\beta\in (0,1)$ and $C>0$,
    \[
      \|w\|_{H^1(\mathcal{X})}\leq
      C\|\mathfrak{u}\|^{\beta}_{L^2(V_0 \setminus V_1)}\|w\|_{H^1((-\frac T3, \frac
        T3)\times (V_{-2}\setminus V_3))}^{1-\beta}.
    \]
    Altogether
    \[
      \|\partial_tw\|_{L^2((-\frac T2, \frac T2)\times (N\setminus V)}+\|\mathfrak{u}\|_{L^2
        (N\setminus V)}\leq C\|\mathfrak{u}\|_{L^2(V)}^{\beta}\|w\|_{H^1((-T,T)\times
        N)}^{1-\beta}.
    \]
    This concludes the proof.
  \end{proof}

  \begin{remark}\label{rem:change_exponents}
    By the usual properties of the trace operator, in the statement of
    Lemma \ref{lemma: propagation of smallness ball}, one has
    \[
      \|\mathfrak{u} \|_{L^2(\omega)}\leq \| w\|_{H^3((-T,T) \times
        B(z,e^{\eta}R))};
    \]
    likewise, in the statement of Lemma
    \ref{lem:prop_smallness_compact}, one has
    \[
    \|\mathfrak{u} \|_{L^2(V)}\leq \|w\|_{H^2((-T,T)\times
      N)}.
  \]
  Thus, in both lemmas, one can reduce the value of $\alpha$. This
  fact will be used in Section \ref{sec: compact core and general
    case}, where we will glue different estimates together, requiring
  the two values of $\alpha$ to be identical.
  \end{remark}

\subsection{Propagation of smallness in funnels and cusps}
\label{sec: funnel}

This subsection is dedicated to proving that the
propagation of smallness result of Lemma \ref{lemma: propagation of
  smallness ball} extends to funnels and cusps as long as
$\omega$ satisfies the thickness condition \eqref{ass: renormalized
  thickness condition}. In the simpler case of a funnel, the proof
consists in covering the funnel (up to a compact remainder) with balls
of a fixed radius. The balls will spill on a finite number of images
of the funnel by elements of its corresponding Fuchsian group, as
depicted in Figure \ref{fig:cover_balls}.

We recall here the principles behind ball covering in a metric
space. Let $E\subset \mathbb{H}$ be an arbitrary set. To cover $E$
with balls of radius $R$, pick a \emph{saturated $R$-separated set}
$(x_i)_{i\in \mathcal{I}}$, that is, among the subsets of $E$ for
which any two distinct points $x_i\neq x_j$ satisfy
$\dist_{\mathbb{H}}(x_i,x_j)\geq R$, one that is maximal for the
inclusion. By design, the family $(x_i)_{i\in \mathcal{I}}$ satisfies
the following properties:
\begin{itemize}
\item This family is locally finite; in particular (because $\mathbb{H}$ is paracompact) $\mathcal{I}$ is either
  finite or countably infinite.
\item $E$ is contained in $\cup_{i\in
    \mathcal{I}}B_{\mathbb{H}}(x_i,R)$, which itself is contained in
  the $R$-neighborhood of $E$.
\item Given $i\in \mathcal{I}$, the intersection number
  \[
    \#\{j\in \mathcal{I},B_{\mathbb{H}}(x_i,R)\cap
    B_{\mathbb{H}}(x_j,R)\neq \emptyset\}=\#\{j\in
    \mathcal{I},\dist_{\mathbb{H}}(x_i,x_j)<2R\}\]
  is bounded from above by
  \[
    \frac{\sinh^2(\frac{5R}{4})}{\sinh^2(\frac
      R4)}=(1+e^{-\frac
      R2}+e^{-R}+e^{-\frac{3R}{2}}+e^{-2R})^2e^{2R}\leq 25e^{2R}.
  \]
  Indeed, the balls of radius $R/2$ centered at the $x_j$'s are
  disjoint, and they are all contained in a ball centered at $x_i$ of
  radius $5R/2$; the volume of a ball of radius $r$ is
  $4\pi\sinh^2(\tfrac r2)$.
\item More generally, for every fixed $r>0$ one has a uniform bound of
  the type
  \begin{equation}\label{eq:ctrl_number_balls}
    \#\{j\in
    \mathcal{I},\dist_{\mathbb{H}}(x_i,x_j)<r\}<C(r,R)
  \end{equation}
\end{itemize} as soon as a volume doubling estimate like \eqref{eq: volume doubling intro} holds.

We are now ready to adapt Lemma \ref{lemma: propagation of smallness ball} to the case of funnels. 
Recall now that a \emph{funnel} $F$ is a ``left half'' of
$\mathbb{H}/\{z\sim e^{\ell}z\}$: a fundamental domain is of the form
\[
  D_F=\{x+iy,x<0<y,1\leq x^2+y^2<e^{\ell}\}
\]
where $\ell$ is the length of the geodesic delimiting $F$.
\begin{lemma}
  \label{lem:propag_smallness_funnels}
  Let $F$ be a hyperbolic funnel and let $R,\delta,T,\eta$ be positive constants. There exists
  a compact set $K\subset F$ and constants $C>0,\alpha\in (0,1)$ such
  that the following is true.
  
  Let $\omega\subset F$ satisfy
  \[
    \forall x\in F, \quad \vol(B(x,R)\cap \omega)\geq \delta
  \]
  and let $\mathfrak{u}:F\to \mathbb{C}$, as well as
  $w:\mathbb{R}\times F \to \mathbb{C}$ satisfy
  \[
    (\partial_t^2 +
    \Delta_\mathbb{H}) w = 0, \quad \partial_t w(t=0) = \mathfrak{u},
    \quad \nabla_\mathbb{H} w(t=0) = 0.\]
  Then
  \begin{equation}
    \label{eq: propagation of smallness, funnel}
    \|\mathfrak{u}\|_{L^2(F\setminus K)} +\|\partial_t
    w\|_{L^2((-\frac{T}{2},\frac{T}{2}) \times (F\setminus K)))} \leq C \|\mathfrak{u} \|_{L^2(\omega)}^{\alpha}\| w\|_{H^3((-T,T) \times F)}^{1-\alpha}.
  \end{equation}
\end{lemma}
\begin{proof}
  To begin with, let us ``unroll'' the functions $\mathfrak{u}$ and
  $w$, from being defined on $F$ and $\mathbb{R}\times F$
  respectively, to $\{z\in \mathbb{H},{\rm Re}(z)<0\}$ and
  $\mathbb{R}\times\{z\in \mathbb{H},{\rm Re}(z)<0\}$, such that they
  are invariant under $z\mapsto e^{\ell}z$. We
  denote these extensions $\widetilde{\mathfrak{u}}$ and
  $\widetilde{w}$. We also extend $\omega$ periodically into $\widetilde{\omega}$.

  Now we cover $D_F$ with balls of radius $R$ as described
  previously. Let $x\in D_F$ and let $L>0$ small enough so that
  $B_{\mathbb{H}}(x,2R)\subset \{z\in \mathbb{H},{\rm
    Im}(z)<L^{-1}\}$. Then, by Lemma \ref{lemma: propagation of smallness ball},
  \begin{equation}
  \|\widetilde{\mathfrak{u}}\|_{L^2(B(x,R))} +\|\partial_t
    \widetilde{w}\|_{L^2((-\frac{T}{2},\frac{T}{2}) \times B(x,R))}
    \leq C \|\widetilde{\mathfrak{u}}
    \|_{L^2(\widetilde{\omega}\cap B(x,R))}^{\alpha}\| \widetilde{w}\|_{H^3((-T,T) \times B(x,2R))}^{1-\alpha}.
  \end{equation}
  We insist on the fact that $C$ and $\alpha$ do not depend on $x$.

  Our goal is now to sum the last estimate over balls which cover $D_F$
  minus a compact set. Letting $(x_j)_{j\in \mathcal{I}}$ be a
  maximal $R$-separated set of $D_F$, there exists a finite subset
  $\check{\mathcal{I}}$ of indices and some small $L > 0$ such that
    \[B_{\mathbb{H}}(x_j,2R) \nsubseteq \{z\in \mathbb{H},{\rm
    Im}(z)<L^{-1}\}, \, i \in \check{\mathcal{I}}.\] Denoting
$\mathcal{J}=\mathcal{I}\setminus \check{\mathcal{I}}$, the balls $\{B(x_j,R),i\in \mathcal{J}\}$ cover $D_F$ except for a bounded neighborhood
$K$ of the boundary $\{{\rm Im}(z)=L^{-1}\}$. One then has,
by the covering property
\[
  \|\widetilde{\mathfrak{u}}\|^2_{L^2(D_F\setminus K)}\leq \sum_{j\in
    \mathcal{J}}\|\widetilde{\mathfrak{u}}\|^2_{L^2(B(x_j,R))}
\]
and, by the bounded intersection number property,
\[
  \sum_{j\in
    \mathcal{J}}\|\widetilde{\mathfrak{u}}\|^2_{L^2(B(x_j,R)\cap
    \widetilde{\omega})}\leq C(R)\|\widetilde{\mathfrak{u}}\|^2_{L^2(V_R\cap
    \widetilde{\omega})}
\]
where $V_R$ is the $R$-neighbourhood of $D_F\setminus K$. In turn,
since $V_R$ intersects a finite number of copies of $D_F$, for every
periodic function $\tilde{v}$ which lifts a square-integrable function $v:D_F\to \mathbb{C}$,
\begin{equation}\label{eq:comparison_L2_lifts_funnel}
  \|v\|^2_{L^2(B(x_j,R))}\leq \|\Tilde{v}\|^2_{L^2(B_\mathbb{H}(x_j,R))}\leq
  C(R)\|v\|^2_{L^2(B(x_j,R))}.
\end{equation}
Thus
  \[
  \sum_{j\in
    \mathcal{J}}\|\widetilde{\mathfrak{u}}\|^2_{L^2(B(x_j,R)\cap
    \widetilde{\omega})}\leq C\|\widetilde{\mathfrak{u}}\|^2_{L^2(D_F\cap
    \widetilde{\omega})}
\]
as well as
\begin{equation*}
  \sum_{j\in \mathcal{J}}\|\widetilde{w}\|_{H^3((-T,T) \times B(x_j,2R))}
  \leq C(R)\|\widetilde{w}\|_{H^3((-T,T) \times D_F)}^2.
\end{equation*}
Identifying $D_F$ with $F$ and $D_F\cap \widetilde{\omega}$ with
$\omega$, by Hölder's inequality, the proof is complete.
\end{proof}
\begin{figure}
  	\centering
	\includegraphics[width=0.9\textwidth]{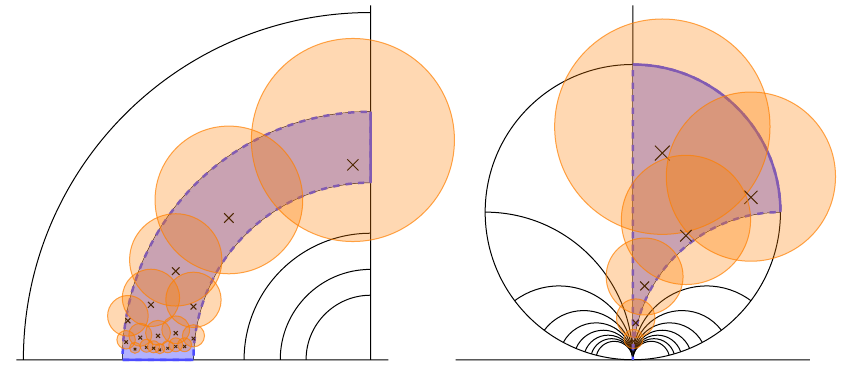}
	\caption{Covering a fundamental domain (blue) of a funnel (left) and
    a cusp (right) with balls of
    fixed radius (orange). The centres of the balls are separated: none
    of the
    centres belong to another ball.\\In the case of a funnel, the
    balls spill over at most four copies of the fundamental domain. In
    the case of a cusp, as the centres of the balls move towards the
    tip, the balls cover more and more copies.}
  \label{fig:cover_balls}
\end{figure}

Our next task is to treat \emph{cusps} in the same way; a cusp is an
``upper half'' of $\mathbb{H}/\{z\sim z+1\}$; a fundamental domain is
\[
  D_C=[-\tfrac 12, \tfrac 12)\times (\ell^{-1},+\infty)
\]
where $\ell$ is the length of the limit horocycle.

With respect
to the proof of Lemma \ref{lem:propag_smallness_funnels}, cusps
present an apparent difficulty: a covering by balls of the fundamental domain of a
cusp must spill over a neighborhood of size $R$ of the domain,
and this neighborhood intersects an arbitrarily large number of
copies of $D_C$. This difficulty is non-essential:
one can bound $L^2$ norms, from above and below, in a way which grows
linearly with the number of intersected domains.

\begin{lemma}
  \label{lem:propag_smallness_cusps}
  Let $\mathcal{C}$ be a hyperbolic cusp and let $R,\delta,T,\eta$ be positive constants. There exists
  a compact set $K\subset \mathcal{C}$ and constants $C>0,\alpha\in (0,1)$ such
  that the following is true.
  
  Let $\omega\subset \mathcal{C}$ satisfy
  \[
    \forall x\in \mathcal{C}, \quad \vol(B(x,R)\cap \omega)\geq \delta
    \vol(B(x,R))
  \]
  and let $\mathfrak{u}:\mathcal{C}\to \mathbb{C}$, as well as
  $w:\mathbb{R}\times \mathcal{C} \to \mathbb{C}$ satisfy
  \[
    (\partial_t^2 +
    \Delta_\mathbb{H}) w = 0, \quad \partial_t w(t=0) = \mathfrak{u},
    \quad \nabla_\mathbb{H} w(t=0) = 0.\]
  Then
  \begin{equation}
    \label{eq: propagation of smallness, cusp}
    \|\mathfrak{u}\|_{L^2(\mathcal{C}\setminus K)} +\|\partial_t
    w\|_{L^2((-\frac{T}{2},\frac{T}{2}) \times (\mathcal{C}\setminus K)))}\leq C \|\mathfrak{u} \|_{L^2(\omega)}^{\alpha}\| w\|_{H^3((-T,T) \times \mathcal{C})}^{1-\alpha}.
  \end{equation}
\end{lemma}
\begin{proof}This proof will repeatedly use balls in both the
  hyperbolic space $\mathcal{H}$ and the funnel $\mathcal{C}$; we use
  subscripts accordingly.
  
  Let $\widetilde{\mathfrak{u}}:\{{\rm Im}(z)>\ell^{-1}\}\to \mathbb{C}$ be a
  horizontally periodic lift of the function $\mathfrak{u}$. Let $z\in
  \mathcal{C}$ and $\widetilde{z}$ be the corresponding point in $D_{\mathcal{C}}$, such that $B_{\mathbb{H}}(z,R)\subset \{{\rm
    Im}(z)>\ell^{-1}\}$. Let $N$ be the number of copies of
  $D_{\mathcal{C}}$ (obtained by horizontal translations by relative
  integers) which intersect $B_{\mathbb{H}}(z,R)\subset \{{\rm
    Im}(z)>\ell^{-1}\}$. Obviously there holds, for every periodic
  function $\widetilde{v}$ lifting a square-integrable function
  $v:D_{\mathcal{C}}\to \mathbb{C}$:
  \begin{equation}\label{eq:ctrl_u_cusp}
    \|\widetilde{v}\|_{L^2(B_\mathbb{H}(\Tilde{z},R))}^2\leq
    N\|v\|^2_{L^2(B_\mathcal{C}(z,R))}.
  \end{equation}
  The trivial bound in the other direction
  \[
    \|v\|^2_{L^2(B_\mathcal{C}(z,R))}\leq \|\widetilde{v}\|_{L^2(B_\mathbb{H}(\Tilde{z},R))}^2
  \]
  is unbalanced, as far as (\ref{eq:comparison_L2_lifts_funnel}) is concerned; let us prove a better bound if $N\geq 3$.
  
  We claim that there exists $c(R)>0$ such that, for $|k|<(N-2)c(R)$ one has \[B_{\mathbb{H}}(z+k,\tfrac
  R2)\subset B_{\mathbb{H}}(z,R),\] that is,
  $\dist_{\mathbb{H}}(z,z+k)<\frac R2$. To this end we recall that
  \[
    B_{\mathbb{H}}(x+iy,R)=B_{\mathbb{R}^2}(x+iy\cosh(R),y\sinh(R));
  \]
  in particular, the definition of $N$ implies that
  \[
    N-2\leq 2{\rm Im}(\widetilde{z})\sinh(R),
  \]
  and now, given $k\in \mathbb{Z}$, one has
  \[
    \cosh(\dist_{\mathbb{H}}(\widetilde{z},\widetilde{z}+k))=1+\frac{k^2}{2{\rm Im}(\widetilde{z})^2}<\cosh(\tfrac{R}{2})
  \]
  if and only if
  \[
    |k|<\sqrt{2}{\rm Im}(\widetilde{z})\sqrt{\cosh(\tfrac R2)-1};
  \]
  in particular if \[|k|<(N-2)\frac{\sqrt{\cosh(\frac
        R2)-1}}{\sqrt{2}\sinh(R)},\] one has the required bound.

  The ball inclusion above implies that
  \begin{equation}\label{eq:ctrl_utilde_cusp}
    2(N-2)c(R)\|v\|^2_{L^2(B_{\mathcal{C}}(z,\frac R2))}\leq
    \|\widetilde{v}\|^2_{L^2(B_{\mathbb{H}}(\widetilde{z},R))},
  \end{equation}
  which is now comparable with the inequality in the other direction
  if $N\geq 3$.

  We can now complete the proof of Lemma
  \ref{lem:propag_smallness_cusps} using
  (\ref{eq:ctrl_u_cusp}-\ref{eq:ctrl_utilde_cusp}) instead of
  (\ref{eq:comparison_L2_lifts_funnel}). Let
    $\widetilde{\omega}$ be the lift of $\omega$ to $\mathbb{H}$. Let us prove first that
    there exists $\delta'>0$ satisfying
    \begin{equation}\label{eq:control_mass_omegatilde}
      \forall x\in \widetilde{D_{\mathcal{C}}}, {\rm
        vol}(B_{\mathbb{H}}(x,2R)\cap \widetilde{\omega})>\delta'.
    \end{equation}
    We let $N(2R,x)$ denote the number of copies of $D_{\mathcal{C}}$
    that are intersected by $B_{\mathbb{H}}(x,2R)$. If $N(2R,x)\leq 2$
    then 
    \[
      \|\mathbbm{1}_{\widetilde{\omega}}\|_{L^2(B_{\mathbb{H}}(x,2R))}^2\geq
      \|\mathbbm{1}_{\omega}\|_{L^2(B_{\mathbb{H}}(x,2R)\cap
        D_{\mathcal{C}})}^2=\vol(B_{\mathcal{C}}(x,2R)\cap \omega);
    \]
    on the other hand 
    \[
      C(R)=\|1\|_{L^2(B_{\mathbb{H}}(x,2R))}^2\leq
      N(2R,x)\|1\|_{L^2(B_{\mathcal{C}}(x,2R))}^2=N(2R,x)\vol(B_{\mathcal{C}}(x,2R))
    \]
    so that, by hypothesis, $\vol(B_{\mathcal{C}}(x,2R)\cap
    \omega)>\delta'=\tfrac{C(R)\delta}{2}$. If now $N(2R,x)\geq 3$ then we use
    (\ref{eq:ctrl_u_cusp}-\ref{eq:ctrl_utilde_cusp}) to obtain
    \begin{align*}
      \|1_{\widetilde{\omega}}\|_{L^2(B_{\mathbb{H}}(x,2R))}^2&\geq
                                                                2(N(2R,x)-2)c(R)\|1_{\omega}\|_{L^2(B_{\mathcal{C}}(x,R))}^2\\
      &\geq 2(N(2R,x)-2)c(R)\delta
        \|1\|_{L^2(B_{\mathcal{C}}(x,R))}^2\\
      &\geq \frac{2(N(2R,x)-2)c(R)\delta}{N(R,x)}
        \|1\|_{L^2(B_{\mathbb{H}}(x,R))}.
    \end{align*}
    Since $N(2R,x)\geq N(R,x)$, we obtain finally a uniform lower
    bound.
	
    We now apply Lemma \ref{lemma: propagation of smallness
      ball} to balls of radius $2R$ centered on points of a
    $\frac{R}{2}$-separated covering $(x_i)_{i\in \mathcal{J}}$ of $D_{\mathcal{C}}$. By
    \eqref{eq:control_mass_omegatilde} one has, uniformly in $j\in
    \mathcal{J}$ except for a finite number of them (those whose balls
    intersect the boundary)
    \[
      \|\widetilde{\mathfrak{u}}\|^2_{L^2(B_{\mathbb{H}}(x_j,R))}\leq
      \|\widetilde{\mathfrak{u}}\|^2_{L^2(B_{\mathbb{H}}(x_j,2R))}\leq
      C\|\widetilde{\mathfrak{u}}\|_{L^2(B_{\mathbb{H}}(x_j,2R)\cap \widetilde{\omega})}^{2\alpha}\|\widetilde{w}\|^{2(1-\alpha)}_{H^3([-T,T]\times
        B_{\mathbb{H}}(x_j,2\eta
        R))}
    \]
    so that, by (\ref{eq:ctrl_u_cusp}-\ref{eq:ctrl_utilde_cusp}),
    \[
      \|\mathfrak{u}\|^2_{L^2(B_{\mathcal{C}}(x_j,R/2))}\leq
      C\|\mathfrak{u}\|_{L^2(B_{\mathcal{C}}(x_j,2R)\cap \omega)}^{2\alpha}\|w\|^{2(1-\alpha)}_{H^3([-T,T]\times
        B_{\mathcal{C}}(x_j,2\eta
        R))};
    \]
    finally, summing over $j$ and applying
    \eqref{eq:ctrl_number_balls} and Hölder's inequality, we obtain
    \[
      \|\mathfrak{u}\|^2_{L^2(\mathcal{C}\setminus \mathcal{K})}\leq
      C\|\mathfrak{u}\|_{L^2(\mathcal{C}\cap \omega)}^{2\alpha}\|w\|^{2(1-\alpha)}_{H^3([-T,T]\times
       \mathcal{C})}.
    \]
  
\end{proof}

\subsection{General surface with hyperbolic ends}
\label{sec: compact core and general case}
To prove Theorem \ref{theo: main result, sufficient part}, it remains
to patch together Lemmas \ref{lem:prop_smallness_compact},
\ref{lem:propag_smallness_funnels}, and
\ref{lem:propag_smallness_cusps}, and to apply them to spectral
cutoffs.

On a smooth Riemannian surface $M$ with a connected compact core and a finite number of
hyperbolic ends (either cusps or funnels), if $\omega$ satisfies
\[
  \forall x\in M, \quad \vol(B(x,R)\cap \omega)\geq \delta
\]
we obtain, by the previous lemmas, the existence of $C>0$
and $\alpha\in (0,1)$ such that for every $\mathfrak{u}:M\to
\mathbb{R}$ and $w:(-T,T)\times M\to \mathbb{R}$ satisfying
\[
  (w,\partial_tw)(0,x)=(0,\mathfrak{u}(x)) \qquad \qquad
  (\partial_t^2+\Delta)w=0
\]
one has
\begin{equation}\label{eq:propag_smallness_global}
  \|\mathfrak{u}\|_{L^2(M)}\leq
  C\|\mathfrak{u}\|^{\alpha}_{L^2(\omega)}\|w\|^{1-\alpha}_{H^3(M\times (-T,T))}.
\end{equation}

Indeed, after using Remark \ref{rem:change_exponents} to fix the same value
of $\alpha$ for Lemmas \ref{lemma: propagation of smallness ball}
and \ref{lem:prop_smallness_compact}, we can decompose $M$ into a
finite set of hyperbolic ends and a compact core, apply Lemmas
\ref{lem:prop_smallness_compact}, \ref{lem:propag_smallness_funnels}, and
\ref{lem:propag_smallness_cusps} to each piece, and then use Hölder's
inequality to control the sum over $M$ as in the end of the proof of
Lemmas \ref{lem:prop_smallness_compact} and
\ref{lem:propag_smallness_funnels}.

It remains to apply this inequality to the case where $\mathfrak{u}$
lies in the range of $\Pi_{\Lambda}$. In this case, by spectral decomposition,
\[
  \|w\|_{H^3(M\times (-T,T))}\leq
  C_0\Lambda^3Te^{\Lambda T}\|\mathfrak{u}\|_{L^2(M)}
\]
and, re-injecting in the left-hand-side, we obtain \eqref{eq: spec}.

\section{Spectral estimates imply thickness}
\label{sec: proof necessary condition}

In this section, we prove that thickness is necessary to spectral estimates on manifolds with Ricci curvature bounded from below. Theorem \ref{theo: main result, necessary part} generalizes the necessity of Logvinenko and Sereda's thickness condition \cite{logvinenkosereda1974equivalent} (see for example \cite[Chapter 3.4]{havinjoricke1995book_uncertaintyprinciple}). The proof is a generalization of those of \cite[Lemma 2.5]{wangwangzhangzhang2019spectralineq_heat} and \cite[Section 2.3]{rouveyrol2023spectralH2}. It is based on the fact that spectral estimates imply an observability inequality for solutions of the heat equation. The idea is then to test this observability inequality with heat kernels. The Ricci curvature lower bound implies Gaussian estimates from above and below for the heat kernel as well as estimates on the volume of balls, both of which are used in the proof.

We start with a few key inequalities on manifolds with Ricci curvature bounded from below. The first one is a volume doubling property which is a consequence of \cite{cheegergromovtaylor1982kernelestimates} (see also \cite[Theorem 1.5]{hebey1996sobolev}, \cite[Theorem 4.19]{gallothulinlafontaine2004book}): \begin{lemma}[Volume doubling]
    There exists a positive constant $C_D(n,K) > 0$ such that for any $z \in M$ and every $r > 0$, one has \begin{equation}
        \label{eq: volume doubling}
        \vol(B_z(2r)) \leq C_D e^{C_D r} \vol(B_z(r)).
    \end{equation} Iterating \eqref{eq: volume doubling} gives for any $n \in \mathbb{N}$: \begin{equation}
        \label{eq: volume dilation}
        \vol(B_z(2^n r)) \leq (C_D)^n \exp\left(C_D (2^n - 1) r\right) \vol(B_z(r)).
    \end{equation}
\end{lemma}

The following is a consequence of the volume doubling estimate: \begin{lemma}[Quotient of Gaussian integrals]
    \label{lemma: quotient of gaussians}
    Consider two positive numbers $\alpha, \beta$. Then \begin{equation}
        \label{eq: quotient of Gaussians lower bound}
        \inf_{z_0 \in M} \frac{\int_M e^{-\alpha d(z, z_0)^2} d\vol(z)}{\int_M e^{-\beta d(z, z_0)^2} d\vol(z)} > 0.
    \end{equation}
\end{lemma}

\begin{proof}[Proof of Lemma \ref{lemma: quotient of gaussians}]
The idea of the proof is to compare the integral of a Gaussian centered around a point $z_0$ with the volume of the ball $B_{z_0}(1)$. We start by bounding the numerator from below:

$$\int_M e^{-\alpha d(z, z_0)^2} d\vol(z) \geq e^{-\alpha} \vol(B_{z_0}(1)).$$

We now bound the denominator from above by decomposing the integral on a sum of integrals over annuli and using the volume dilation estimate \eqref{eq: volume dilation}:

$$\begin{aligned}
    &\int_M e^{-\beta d(z,z_0)^2} d\vol(z) \\
    &\leq \int_{B_{z_0}(1)} e^{-\beta d(z,z_0)^2} d\vol(z) + \sum_{n=0}^{+\infty} \int_{B_{z_0}(2^{n+1}) \setminus B_{z_0}(2^n)} e^{-\beta d(z,z_0)^2} d\vol(z)\\
    &\leq \vol(B_{z_0}(1)) + \sum_{n=0}^{+\infty} e^{-\beta 4^n} \vol\left( B_{z_0}(2^{n+1}) \right)\\
    &\leq \vol(B_{z_0}(1)) + \left(\sum_{n=0}^{+\infty} \exp\left(2^{n+1} C_D - 4^n \beta\right) (C_D)^{n+1}\right) \vol(B_{z_0}(1)).
\end{aligned}$$ The sum converges, so that we get the $z_0$-independent lower bound: 
$$\frac{\int_M e^{-\alpha d(z, z_0)^2} d\vol(z)}{\int_M e^{-\beta d(z, z_0)^2} d\vol(z)} \geq \frac{e^{-\alpha}}{1 + \left(\sum_{n=0}^{+\infty} \exp\left(2^{n+1} C_D - 4^n \beta\right) (C_D)^{n+1}\right)}.$$ This proves Lemma \ref{lemma: quotient of gaussians}. Note that the numerator only needs to be an integral of a non-negative, radially decaying function, rather than of a Gaussian.
\end{proof}

Lower bounds on Ricci curvature also ensure the following Gaussian (in terms of the geodesic distance) upper and lower bounds on the heat kernel \cite[Theorem 2.35]{baudoin2022geometricinequalities}: \begin{lemma}[Gaussian bounds for the heat kernel]
    \label{eq: heat kernel bounds}
    Denote $p(t,z,z_0)$ the heat kernel of $M$ at time $t$ between two points $z$, $z_0$ in $M$. There exist constants $C_1, C_2$ depending only on $(n, K)$ such that the following estimates hold:
    \begin{equation}
        \label{eq: heat kernel upper bound}
        p(z, z_0, t) \leq \frac{C_1}{\vol(B_{z_0}(\sqrt{t}))} \exp\left(-\frac{d(z,z_0)^2}{5t} + C_2(t + d(z, z_0)^2) \right),
    \end{equation}
    \begin{equation}
        \label{eq: heat kernel lower bound}
        p(z, z_0, t) \geq \frac{C_1^{-1}}{\vol(B_{z_0}(\sqrt{t}))} \exp\left(-\frac{d(z,z_0)^2}{3t} - C_2(t + d(z, z_0)^2) \right).
    \end{equation}
\end{lemma}

For proof of these geometric inequalities, we refer to \cite{baudoin2022geometricinequalities} and the references therein, in particular \cite{liyau1986parabolickernel,grigor'yan2009book}.

We now consider a Borel subset $\omega \subset M$ satisfying the spectral estimate, which we recall: \begin{equation}
    \label{eq: spectral estimate-necessary}
    \forall u \in L^2, \forall \Lambda > 0, \quad \|\Pi_\Lambda u\|_{L^2(M)} \leq Ce^{C\Lambda} \|\Pi_\Lambda u\|_{L^2(\omega)}.
\end{equation}

The spectral estimate \eqref{eq: spectral estimate-necessary} implies the following final-time observability estimate from $(0,T) \times \omega$ (see \cite{miller2010directlebeaurobbiano,phungwang2013observability,wangwangzhangzhang2019spectralineq_heat,burqmoyano2023propagationheat} for proof and refined statements, in particular concerning the time set): \begin{lemma}[Observability]
        \label{lemma: observability}
        Assume $\omega$ satisfies the spectral estimate property \eqref{eq: spectral estimate-necessary} and consider $u(t, z)$ a solution of the heat equation \begin{equation}
            \label{eq: heat equation}
            \begin{aligned}
                \partial_t u - \Delta_g u &= 0 \text{ in } \mathbb{R}^+_t \times M,\\
                u|_{t=0} &= u_0 \text{ in } M.
            \end{aligned}
        \end{equation} for some $u_0 \in L^2(M)$. For any $T > 0$, there exists some constant $C_{obs}(T) > 0$ such that \begin{equation}
            \label{eq: observability}
            \int_M |u(T, z)|^2 d\vol(z) \leq C_{obs}(T) \int_0^T \int_\omega |u(s, z)|^2 d\vol(z) ds.
        \end{equation}
\end{lemma} A proof of (a stronger version of) Lemma \ref{lemma: observability} on the flat space can be found in \cite[Lemmas 2.3 and 2.4]{wangwangzhangzhang2019spectralineq_heat}. It adapts easily to the manifold case, as explained in \cite[Appendix B]{rouveyrol2023spectralH2}.

Fix $t_K > 0$ a time such that \[-\beta_K = -\tfrac{2}{5 t_K} + 2C_2 < -4 C_D,\] for example $t_K = \frac{1}{10(C_2 + 2C_D)}$. Note that $t_K$ depends only on the dimension and curvature bound of the manifold $M$. We fix some $z_0 \in M$ and apply the observability inequality \eqref{eq: observability} at time $T = \tfrac{t_K}{2}$ for the initial condition \begin{equation}
    \label{eq: heat initial condition}
    u_0(z) = p\left(z, z_0, \frac{t_K}{2}\right), \quad z \in M.
\end{equation} By the semi-group property for the heat equation, we get \begin{equation}
    \label{eq: heat solution} u(s, z) = p\left(z, z_0, s + \frac{t_K}{2}\right), \quad s \geq 0, z \in M,
\end{equation} thus \begin{equation}
    \label{eq: observability for heat kernel}
    \int_M |p\left(z, z_0, t_K\right)|^2 d\vol(z) \leq C_{obs} \int_0^{\frac{t_K}{2}} \int_\omega \left|p\left(z, z_0, \frac{t_K}{2} + s\right)\right|^2 d\vol(z) ds.
\end{equation}

We bound the left-hand-side from below using \eqref{eq: heat kernel lower bound}:
$$\begin{aligned}
    \int_M \left|p\left(z, z_0, t_K\right)\right|^2 d\vol(z) &\geq \frac{C_1^{-2}}{\vol(B_{z_0}(\sqrt{t_K}))^2} \int_M \exp^{-2C_2 t_K - \alpha_K d(z, z_0)^2} d\mu(z) \\
    &\geq \frac{m_K}{\vol(B_{z_0}(\sqrt{t_K}))^2} \int_M e^{-\alpha_K d(z, z_0)^2} d\mu(z)
\end{aligned}$$ with $m_K$ and $\alpha_K = 2\left(C_2 + \frac{1}{3t_K}\right)$ two constants depending on the dimension and the curvature lower bound. In the rest of the proof, constants indexed by a $K$ are only dependent on those geometric features. They do not depend on $z_0$ or the parameter $R$ that will be introduced later on.

We now bound the right-hand-side from above using the Gaussian upper bound \eqref{eq: heat kernel upper bound}. We then introduce some radius $R > 0$, which will be fixed independently of $z_0$ at the end of the proof. We split the integral over $\omega$ into its $B_{z_0}(R)$ and $B_{z_0}(R)^c$ parts:
$$\begin{aligned}
    &\int_0^{t_K/2} \int_\omega \left|p\left( z, z_0, t+s \right)\right|^2 d\vol(z) ds\\
    &\leq \int_0^{t_K/2} \int_\omega \frac{C_1^2}{\vol\left(B_{z_0}\left(\sqrt{\frac{t_K}{2} + s}\right)\right)^2} e^{C_2\left(t_K + 2s\right)} \exp^{-\beta_K d(z, z_0)^2} d\vol(z)\\
    &\leq \frac{C_ 1^2 e^{2C_2 t_K}}{\vol\left(B_{z_0}\left(\sqrt{\frac{t_K}{2}}\right)\right)^2} \left( \int_{\omega \cap B_{z_0}(R)} e^{-\beta_K d(z, z_0)^2} d\vol(z) + \int_{\omega \cap B_{z_0}(R)^c} e^{-\beta_K d(z, z_0)^2} d\vol(z)\right)\\
    &\leq \frac{M_K}{\vol\left(B_{z_0}\left(\sqrt{\frac{t_K}{2}}\right)\right)^2} \left(\vol(\omega \cap B_{z_0}(R)) + e^{-\frac{\beta_K}{2}R^2} \int_{\omega \cap B_{z_0}(R)^c} e^{-\frac{\beta_K}{2}d(z, z_0)^2} d\vol(z) \right) \\
    &\leq \frac{M_K}{\vol\left(B_{z_0}\left(\sqrt{\frac{t_K}{2}}\right)\right)^2} \left(\vol(\omega \cap B_{z_0}(R)) + e^{-\frac{\beta_K}{2}R^2} \int_M e^{-\frac{\beta_K}{2}d(z, z_0)^2} d\vol(z) \right).
\end{aligned}$$

Plugging the lower and upper bounds in \eqref{eq: observability for heat kernel} yields \begin{equation}
    \label{eq: volume and gaussian integrals}
    \int_M e^{-\alpha_K d(z, z_0)^2} d\mu(z) \leq M'_K \left(\vol(\omega \cap B_{z_0}(R)) + e^{-\frac{\beta_K}{2}R^2} \int_M e^{-\frac{\beta_K}{2}d(z, z_0)^2} d\vol(z) \right).
\end{equation} where $$M'_K = \frac{M_K}{m_k} \sup_{z_0 \in M} \frac{\vol(B_{z_0}(\sqrt{t_K}))^2}{\vol\left(B_{z_0}\left(\sqrt{\frac{t_K}{2}}\right)\right)^2}$$ can be chosen independently of $z_0$ thanks to the volume doubling estimate \eqref{eq: volume doubling}.

We conclude the proof of Theorem \ref{theo: main result, necessary part} using Lemma \ref{lemma: quotient of gaussians}. By the lemma, for $\alpha = \alpha_K$ and $\beta = \frac{\beta_K}{2}$, it is possible to find some $R = R_K > 0$ such that $$e^{-\frac{\beta_K}{2} R^2} < \frac{1}{2} \inf_{z_0 \in M} \frac{\int_M e^{-\alpha_K d(z, z_0)^2} d\vol(z)}{\int_M e^{-\frac{\beta_K}{2} d(z, z_0)^2} d\vol(z)}.$$ For such an $R$, \eqref{eq: quotient of Gaussians lower bound} yields \begin{equation}
    \int_M e^{-\alpha_K d(z, z_0)^2} d\mu(z) \leq 2 M'_K \left(\vol(\omega \cap B_{z_0}(R)) \right),
\end{equation} hence we recover the thickness condition \begin{equation}
    \delta \vol(B_{z_0}(R)) \leq \vol(\omega \cap B_{z_0}(R))
\end{equation} for $$\delta = \frac{1}{2M'_K} e^{-\alpha_K R^2}.$$ $\delta$ depends only on $R$, the dimension $n$ of the manifold, and the curvature lower bound.

We finish this section by noting that the proof that observability implies a thickness assumption may generalize to other semigroups $e^{-tH}$. The only assumptions needed are the volume doubling estimate, which depends on the metric rather than the operator considered, and an integral kernel of the propagator satisfying a Gaussian pointwise bound from above and a positive radially decaying bound from below.

\appendix

\section{Statement of propagation of smallness inequalities used in the paper}
\label{sec: appendix propagation of smallness results}

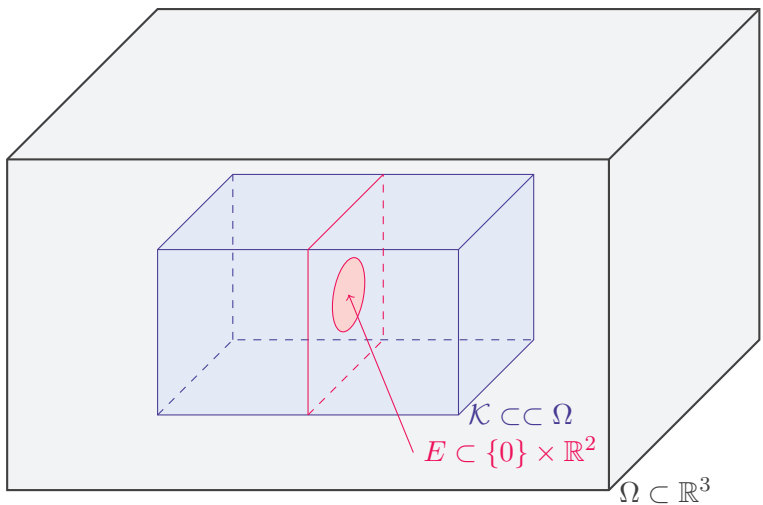
\begin{figure}[t]
    \centering
    \begin{tikzpicture}[scale=2.]
        \coordinate (A) at (0.,0.); 
        \coordinate (B) at (4.,0.);
        \coordinate (C) at (4.,2.2);
        \coordinate (D) at (0.,2.2);
        \coordinate (E) at (1.,1.); 
        \coordinate (F) at (5.,1.);
        \coordinate (G) at (5.,3.2);
        \coordinate (H) at (1.,3.2);
        
        \filldraw[color=darkgray,fill=Gray!10, thick] (A) -- (B) -- (F) -- (G) -- (H) -- (D) -- cycle;
        \draw (B) node[right] {\color{darkgray} $\Omega \subset \mathbb{R}^3$};
        \draw[color=darkgray,thick] (D) -- (C) -- (G);
        \draw[color=darkgray,thick] (C) -- (B);
    
        \coordinate (A1) at (1.,0.5); 
        \coordinate (B1) at (3.,0.5);
        \coordinate (C1) at (3.,1.6);
        \coordinate (D1) at (1.,1.6);
        \coordinate (E1) at (1.5,1.); 
        \coordinate (F1) at (3.5,1.);
        \coordinate (G1) at (3.5,2.1);
        \coordinate (H1) at (1.5,2.1);
        
        \filldraw[color=BlueViolet,fill=NavyBlue!10] (A1) -- (B1) -- (F1) -- (G1) -- (H1) -- (D1) -- cycle;
        \draw (B1) node[right] {\color{BlueViolet} $\mathcal{K} \subset\subset \Omega$};
        \draw[BlueViolet] (D1) -- (C1) -- (G1);
        \draw[BlueViolet] (C1) -- (B1);
        \draw[BlueViolet,dashed] (A1) -- (E1) -- (F1);
        \draw[BlueViolet,dashed] (E1) -- (H1);

        \coordinate (J) at (2.,0.5); 
        \coordinate (K) at (2.5,1.); 
        \coordinate (L) at (2.5,2.1); 
        \coordinate (M) at (2.,1.6); 
        \coordinate (O) at (2.27,1.3); 
        \coordinate (P) at (2.7,0.25);

        \draw[dashed,OrangeRed] (J) -- (K) -- (L);
        \draw[OrangeRed] (J) -- (M) -- (L);
        \filldraw[color=OrangeRed,fill=OrangeRed!20,rotate=-10] (O) ellipse (0.1 and 0.25);
        \draw[OrangeRed,->] (P) -- (O);
        \draw (P) node[right] {\color{OrangeRed} $E \subset \{0\} \times \mathbb{R}^2$};
    \end{tikzpicture}
    \label{fig: Logunov--Malinnikova illustrated}
    \caption{An example of setting where the Logunov--Malinnikova inequality allows to propagate the smallness of gradients from a $2-$dimensional set $E$ to a larger $3-$dimensional box $\mathcal{K}$. The auxiliary $t$ introduced in the paper can be interpreted as the horizontal variable.}
\end{figure}

The goal of this appendix is to recall the different propagation of smallness inequalities that are used throughout the paper. We first present the Logunov--Malinnikova theorem on propagation of smallness for gradients of harmonic functions \cite[Theorem 5.1]{logunovmalinnikova2018qtttve_propag_smallness}, which is used crucially in charts in Section \ref{sec: propagation of smallness}.

\begin{theorem}[\cite{logunovmalinnikova2018qtttve_propag_smallness}]
\label{theo: logunov-malinnikova}Let $d\in \mathbb{N}$. There exists $c\in
(0,1)$ such that the following is true.

Let $\Lambda_1,\Lambda_2,m,\rho,\delta$ be positive constants and let
$\Omega$ be a bounded domain in $\mathbb{R}^{d+1}$. There
exist $C>0$ and $\gamma\in (0,1)$ such that, for every solution of an elliptic equation in divergence form \begin{equation}
        \Div(A\nabla v) = 0 \text{ over } \Omega
    \end{equation} where $A = (a_{i,j}(x))_{i,j}$ is a symmetric,
    uniformly elliptic matrix with Lipschitz entries: $$\Lambda_1^{-1}
    |\zeta|^2 \leq \langle A \zeta, \zeta \rangle \leq \Lambda_1
    |\zeta|^2, \quad |a_{i,j}(x) - a_{i,j}(y) | \leq \Lambda_2
    |x-y|,$$ and for every sets
    $E, \mathcal{K} \subset \Omega$
    satisfying $$C_\mathcal{H}^{d-c+\delta}(E) > m, \quad \dist(E,
    \partial\Omega) > \rho, \quad \dist(\mathcal{K}, \partial\Omega) >
    \rho,$$ one has $$\sup_{\mathcal{K}} |\nabla v| \leq C \sup_{E}|\nabla v|^\gamma \sup_{\Omega}|\nabla v|^{1-\gamma}.$$ 
\end{theorem} In particular, the result holds for $\delta = c$ \emph{ie} if $E$ is a set of Hausdorff dimension $d$. An illustration of this result is presented if Figure \ref{fig: Logunov--Malinnikova illustrated} in the case $d=2, \delta=c$.

In Section \ref{sec: compact core and general case}, which is dedicated to bounds on the compact core of a manifold, we use Carleman inequality-based propagation of smallness from open sets in the spirit of Lebeau and Robbiano \cite{lebeaurobbiano1995controlechaleur}. The first of these results can be found in the textbook \cite[Theorem~9.1]{lerousseaulebeaurobbiano2022carleman2}
(see also \cite[Lemma~5.3]{lerousseaulebeaurobbiano2022carleman1} for the underlying result applied in charts). The proof is based on interior Carleman estimates applied in ordered local charts along paths connecting a finite number of points in the compact core to a point in the observation set. We state it in the specific case $l = 0$:

\begin{theorem}[\cite{lerousseaulebeaurobbiano2022carleman2}, Theorem 9.1]
    \label{theo: Carleman-based propagation of smallness}
    Let $(\mathcal{M}, g)$ be a dimensional, possibly non-compact,
    possibly with boundary, smooth Riemannian manifold. Let $\mathcal{U}, \mathcal{V}$ be two open subsets of $\mathcal{M}$ such that $\mathcal{U}$ is connected, $\overline{\mathcal{U}} \cap \partial\mathcal{M} = \emptyset$ and $\mathcal{U} \subset\subset \mathcal{V} \subset\subset \mathcal{M}$.

    Let $\mathcal{X}$ be an open subset of $\mathcal{U}$. There exist
    $C > 0$ and $\beta \in (0,1)$ such that, for every $w\in H^1(\mathcal{M})$, \begin{equation}
        \label{eq: Carleman-based propagation of smallness}
        \|w\|_{H^1(\mathcal{U})} \leq C  \left( \|\Delta_gw\|_{L^2(\mathcal{V})} + \|w\|_{L^2(\mathcal{X})} \right)^\beta\|w\|_{H^1(\mathcal{V})}^{1-\beta}.
    \end{equation}
\end{theorem}

We will also need a result of propagation of smallness from the
boundary in the Dirichlet situation. We present here a specialisation
of a more general result in this spirit.
\begin{theorem}[\cite{lerousseaulebeaurobbiano2022carleman2}, Theorem 9.3]\label{theo: Carleman boundary}
  Let $(\mathcal{M},g)$ be a possibly non-compact Riemannian manifold,
  with smooth boundary $\partial \mathcal{M}$. Let
  $\mathcal{U}\subset\subset \mathcal{V}\subset \subset \mathcal{M}$
  be open, and suppose that
  $\mathcal{U}$ is connected. Let $\mathcal{X}$ be a non-empty open subset of
  $\mathcal{U}\cap \partial \mathcal{M}$.

  There exist $C>0$ and $\beta\in (0,1)$ such
  that, for all $w\in H^1_0(\mathcal{M})$,
  \begin{equation}
    \label{eq:Carleman_boundary}
    \|w\|_{H^1(\mathcal{U})}\leq C(\|\Delta_gw\|_{L^2(\mathcal{V})}+\|\partial_nw\|_{L^2(\mathcal{X})})^{\beta}\|w\|^{1-\beta}_{H^1(\mathcal{V})}.
  \end{equation}
\end{theorem}




\bibliographystyle{alpha}
\bibliography{bibli}

\end{document}